\documentclass[11pt,reqno]{amsart}
\usepackage[usenames, dvipsnames]{color}
\usepackage{amsmath, amsthm, amssymb, mathrsfs}
\usepackage{epstopdf}
\usepackage{enumitem}
\usepackage{mathtools}
\usepackage{tikz-cd}
\usepackage[utf8]{inputenc}
\usepackage[english]{babel}
\usepackage{fancyhdr}
\usepackage[margin=1.5in]{geometry}
\usepackage{relsize}
\usepackage{hyperref}

\emergencystretch=\maxdimen
\newtheorem{theorem}{Theorem} 
\newtheorem{lemma}{Lemma}     
\newtheorem{corollary}{Corollary}
\newtheorem{proposition}{Proposition}

\newtheorem{definition}{Definition}

\newtheorem*{theorem*}{Theorem}
\newtheorem*{lemma*}{Lemma}
\newtheorem*{proposition*}{Proposition}
\newtheorem*{corollary*}{Corollary}

\newtheorem*{definition*}{Definition}
\newtheorem*{conjecture*}{Conjecture}

\theoremstyle{definition}
\newtheorem{remark}{Remark}
\newtheorem*{remark*}{Remark}

\newcommand{\C}{\mathbb{C}}
\newcommand{\Q}{\mathbb{Q}}

\newcommand{\Z}{\mathbb{Z}}
\newcommand{\PP}{\mathbb{P}}
\newcommand{\OO}{\mathcal{O}}
\newcommand{\DD}{\mathcal{D}}
\newcommand{\ZZ}{\mathcal{Z}}
\newcommand{\WW}{\mathcal{W}}
\newcommand{\QQ}{\mathcal{Q}}
\newcommand{\XX}{\mathcal{X}}
\newcommand{\LL}{\mathcal{L}}
\newcommand{\HHom}{\mathcal{H}om}

\newcommand{\ch}{\textnormal{ch}}
\newcommand{\rk}{\textnormal{rk}}
\newcommand{\td}{\textnormal{td}}
\newcommand{\Ext}{\textnormal{Ext}}
\newcommand{\Quot}{\textnormal{Quot}}
\newcommand{\Hilb}{\textnormal{Hilb}}
\newcommand{\Pic}{\textnormal{Pic}}
\newcommand{\Tan}{\textnormal{Tan}}
\newcommand{\Obs}{\textnormal{Obs}}
\newcommand{\Hom}{\textnormal{Hom}}
\newcommand{\vir}{\textnormal{vir}}

\renewcommand{\dim}{\textnormal{dim}}
\newcommand{\kod}{\textnormal{kod}}
\newcommand{\vd}{\textnormal{vd}}
\newcommand{\red}{\textnormal{red}}

\newcommand{\SW}{\textnormal{SW}}

\newcommand{\id}{\textnormal{id}}

\usepackage{setspace}
\begin{document}
\title[Virtual $\chi_{-y}$-genus of Quot schemes]{Virtual $\chi_{-y}$-genera of Quot schemes on surfaces}

\author{Woonam Lim}
\address{Department of Mathematics, University of California, San Diego}
\email{w9lim@ucsd.edu}

\begin{abstract}
This paper studies the virtual $\chi_{-y}$-genera of Grothendieck's Quot schemes on surfaces, thus refining the calculations of the virtual Euler characteristics by Oprea-Pandharipande. We first prove a structural result expressing the equivariant $\chi_{-y}$-genera of Quot schemes universally in terms of the Seiberg-Witten invariants. The formula is simpler for curve classes of Seiberg-Witten length $N$, which are defined in the paper. By way of application, we give complete answers in the following cases: 
\begin{enumerate}
\item[(i)] arbitrary surfaces for the zero curve class,
\item[(ii)] relatively minimal elliptic surfaces for rational multiples of the fiber class,
\item[(iii)] minimal surfaces of general type with $p_g>0$ for any curve classes.
\end{enumerate}
Furthermore, a blow up formula is obtained for curve classes of Seiberg-Witten length $N$. As a result of these calculations, we prove that the generating series of the virtual $\chi_{-y}$-genera are given by rational functions for all surfaces with $p_g>0$, addressing a conjecture of Oprea-Pandharipande. In addition, we study the reduced $\chi_{-y}$-genera for $K3$ surfaces and primitive curve classes with connections to the Kawai-Yoshioka formula. \end{abstract}

\maketitle
\tableofcontents

\section{Introduction}
\subsection{Overview} 
\noindent
It has been a prominent and successful theme in algebraic geometry to study invariants of moduli spaces. Keeping with this theme, we compute virtual invariants of Quot schemes on surfaces. For moduli spaces of sheaf-like objects on surfaces, the Euler characteristic has attracted the most interest compared to other invariants. It appears, for instance,  in the $S$-duality conjecture of \cite{VW}. Recently \cite{TT1,TT2} proposed a mathematical definition of the Vafa-Witten invariants as the virtual Euler characteristic of the moduli of Higgs bundles using $\C^*$-localization. 

Motivated by developments in Vafa-Witten theory, Oprea and Pandharipande \cite{OP} studied virtual Euler characteristics of Quot schemes on surfaces. Complete formulas for quotients of dimension 0 were obtained. Dimension 1 quotients on minimal surfaces of general type with $p_g>0$ were studied with connections to the Seiberg-Witten invariants. The reduced invariants of $K3$ surfaces were proven to be equal to the stable pair invariants, hence given by the Kawai-Yoshioka formula. Based on these computations, \cite{OP} conjectured the rationality of the generating series of virtual Euler characteristics in general.

The virtual Euler characteristic has a natural refinement, namely the virtual $\chi_{-y}$-genus. In \cite{GK}, G\"ottsche and Kool gave a precise conjecture for the virtual $\chi_{-y}$-genera of the moduli spaces of sheaves on surfaces with $p_g>0$; these are the instanton contributions to the $K$-theoretic Vafa-Witten invariants studied in \cite{T2}. In this paper, we study the virtual $\chi_{-y}$-genera of the Quot schemes on surfaces.

For surfaces, relations between various enumerative theories have been extensively studied, see for instance \cite{CK,DKO,GNY1,M,Ta1,Ta2} and others.  In particular, \cite{OP} observed the appearance of the Seiberg-Witten invariants in the Quot scheme theory of surfaces of general type with $p_g>0$. In this paper, we clarify the relationship between the Quot scheme invariants and the Seiberg-Witten invariants for arbitrary surfaces. Indeed Theorem \ref{tuniv} below gives a universal formula for the Quot scheme invariants in terms of the Seiberg-Witten invariants and cohomological data.  Along the way, we remove several technical assumptions present in \cite{OP} and extend the calculations to several new settings. Just as in \cite{OP}, virtual localization \cite{GP} and cosection localization \cite{KL} play an important role. We also make use of new ingredients, most notably the recent calculation of the virtual fundamental classes of the nested Hilbert schemes in \cite{GSY} and \cite{GT}, and the work of \cite{DKO} on Seiberg-Witten invariants. 

The universal formula of Theorem \ref{tuniv} simplifies for curve classes of Seiberg-Witten length $N$. See Definition \ref{dSW} below for terminology. In this context, the answer exhibits a multiplicative structure as explained in Theorem \ref{tmult}. By using multiplicativity, we find complete formulas, given by rational functions, in the following cases: 
\begin{enumerate}
\item[(i)] arbitrary surfaces for the zero curve class,
\item[(ii)] relatively minimal elliptic surfaces for rational multiples of the fiber class,
\item[(iii)] minimal surfaces of general type with $p_g>0$ for any curve classes.
\end{enumerate}

Furthermore, we prove a blow up formula for curve classes of Seiberg-Witten length $N$. This shares some structural features of the conjectural blow up formula of the virtual $\chi_{-y}$-genera of the moduli of sheaves in \cite{GK}. Using the blow up formula and the calculations in (i)--(iii) above, we establish that the generating series of virtual $\chi_{-y}$-genera are given by rational functions for surfaces with $p_g>0$, addressing a conjecture in \cite{OP}.

In a slightly different direction, the reduced invariants of $K3$ surfaces with a primitive curve class are studied via comparison to the stable pair invariants. The answer is governed by the Kawai-Yoshioka formula of \cite{KY}. 

\subsection{Virtual invariants of moduli spaces}\label{sInv}
\noindent
For a smooth projective variety $M$ over $\C$, the Euler characteristic is a basic topological invariant. The Poincar\'e\,-Hopf theorem relates this topological invariant to the integration
\[ e(M)=\int_{[M]} c(T_M).
\]
A refinement of the Euler characteristic is given by the Hirzebruch $\chi_{-y}$-genus
\[\chi_{-y}(M):=\chi(M, \Lambda_{-y}\Omega_M).
\]
The value at $y=1$ recovers the Euler characteristic. Here we denote
\[\Lambda_{-y}E:=\sum_{n=0}^{\rk (E)} [\Lambda^{\, n} E] (-y)^n\in K^0(M)[y] 
\]
for any vector bundle $E$ over $M$. For further use, we also set
\[S_{y}E:=\sum_{n=0}^\infty [S^n E] y^n\in K^0(M)[[y]].
\]
The Hirzebruch-Riemann-Roch theorem relates the Hirzebruch $\chi_{-y}$-genus to the integration
\[ \chi_{-y}(M)=\int_{[M]} \mathcal{X}_{-y}(T_M)
\]
where
\[\mathcal{X}_{-y}(E):=\ch(\Lambda_{-y}E^\vee)\td(E)\]
is a multiplicative characteristic class of vector bundles $E$ over $M$. The integral formulations for these invariants imply that they are constant under smooth deformations. 

There are virtual analogues of the above invariants that apply to moduli spaces endowed with 2-term perfect obstruction theories. We briefly explain the virtual $\chi_{-y}$-genus, following \cite{FG}. Let $M$ be a projective scheme. A 2-term perfect obstruction theory on $M$ is a perfect complex
\[E^\bullet=[E^{-1}\rightarrow E^0]\]
together with a morphism in derived category 
\[\phi: E^\bullet\rightarrow \tau_{\geq -1} (\mathbb{L}_M)\]
such that $h^0(\phi)$ is an isomorphism and $h^{-1}(\phi)$ is a surjection. Denote the dual complex of $E^\bullet$ by \[E_\bullet=[E_0\rightarrow E_1].\] We define the virtual tangent bundle, virtual cotangent bundle and virtual dimension as 
\begin{align*}
T^\vir_M&:=E_0-E_1\ \ \in K^0(M),\\
\Omega^\vir_M&:=E^0-E^{-1} \in K^0(M),\\
\vd &:=\rk(T^\vir_M)=\rk(E_0)-\rk(E_1).
\end{align*}
By the construction of \cite{BF,LT}, there exist a virtual fundamental class 
\[[M]^\vir\in A_\vd(M).\] 
The virtual $\chi_{-y}$-genus is defined as the virtual holomorphic Euler characteristic
\begin{equation}\label{chiy}
\chi_{-y}^\vir(M):=\chi^\vir(M, \Lambda_{-y}\Omega_M^\vir)
\end{equation}
where 
\[\Lambda_{-y} \Omega_M^\vir
=\Lambda_{-y} \big(E^0-E^{-1}\big):=\big(\Lambda_{-y}E^0\big)\big(S_{y}E^{-1}\big)\in K^0(M)[[y]].
\]
The virtual Riemann-Roch theorem of \cite{FG} expresses (\ref{chiy}) as an integral
\[\chi_{-y}^\vir(M)=\int_{[M]^\vir} \mathcal{X}_{-y}(T^\vir_M).\]
By taking $y=1$, we recover the virtual Euler characteristic $e^\vir(M)$ of $M$. Therefore we can think of the virtual $\chi_{-y}$-genus as a refinement of the virtual Euler characteristic.

\subsection{Deformation theory of Quot schemes}\label{sdef}
\noindent
Let $X$ be a smooth projective variety over $\C$ and fix a vector $v\in H^{*}(X,\Q)$ to encode the topological type of sheaves on $X$. Let $\Quot_X(\C^N, v)$ be the Quot scheme parametrizing short exact sequences
\[0\rightarrow S \rightarrow \C^N\otimes \OO_X \rightarrow Q\rightarrow 0\quad \text{such that}
\quad \ch(Q)=v.\]
The natural tangent obstruction theory of the Quot scheme can be summarized as
\begin{align*}
&\Tan=\Hom(S,Q), \\
&\Obs=\Ext^1(S,Q),\\
&\Obs^{\geq 2}=\Ext^{\geq 2}(S,Q). 
\end{align*}
If the higher obstruction spaces vanish at all points in the Quot scheme, then the Quot scheme carries a 2-term perfect obstruction theory whose virtual tangent bundle is point-wise given by
\[
T^\vir_{\Quot}\Big\vert_{\mathsf{pt}}=\Ext_X^\bullet(S,Q).\]

There are at least two cases where all higher obstruction spaces vanish. First, if $X$ is a smooth curve, the vanishing follows by dimension reasons. This was observed by \cite{CFK} and \cite{MO}. Second, if $X$ is a smooth surface and $v$ is a Chern character of torsion sheaves, \cite{MOP} observed that 
\[\Ext^2(S,Q)=\Hom(Q, S\otimes K_X)^\vee=0.\]
We will analyze the resulting invariants in the surface case. 

\subsection{Seiberg-Witten invariants}\label{sSW} 
\noindent
We fix the geometric set up for the Seiberg-Witten invariants. For a more detailed discussion, we refer the reader to \cite{DKO}. Let $X$ be a smooth projective surface and $\beta$ a curve class in $H^2(X,\Z)$. The Hilbert scheme of divisors of topological type $\beta$ will be denoted by $\Hilb_X^\beta$. There exists the universal divisor $\mathcal{D}$ with the ideal exact sequence 
\[0 \rightarrow\mathcal{O}(-\mathcal{D})
\rightarrow \mathcal{O}_{\Hilb_X^\beta \times X}
\rightarrow \mathcal{O}_{\mathcal{D}}\rightarrow0.\]
On the other hand, consider the Picard group $\Pic_X^\beta$ parametrizing line bundles with first Chern class $\beta$. If we fix a point $p\in X$, there exists a Poincar\'e line bundle $\mathcal{P}$ normalized at $p$. We have the Abel-Jacobi map
\[ \mathsf{AJ}:\Hilb_X^\beta\rightarrow \Pic_X^\beta, \quad D\mapsto\mathcal{O}_X(D).\]
Denote the line bundle $\mathcal{O}(\mathcal{D})$ restricted to $\Hilb_X^\beta\times \{p\}$ by $\mathcal{L}$ and its first Chern class by $h:=c_1(\mathcal{L})$.

The Hilbert scheme of divisors possesses a 2-term perfect obstruction theory whose virtual tangent bundle is point-wise given by
\begin{equation}\label{xxx}
T^\vir_{\Hilb_X^\beta}\Big\vert_{\mathsf{pt}}=\Ext_X^\bullet(\mathcal{O}(-D),\mathcal{O}_D)=H^\bullet(\mathcal{O}_D(D)).
\end{equation}
This induces a virtual fundamental class
\[[\Hilb_X^\beta]^\vir\in H_{2\cdot\vd_\beta}(\Hilb_X^\beta), \quad \vd_\beta=\frac{\beta (\beta-K_X)}{2}.
\]
The Poincar\'e invariant of $X$ with insertion $\beta$ is defined in \cite{DKO} as 
\[\sum_{k=0}^\infty \mathsf{AJ}_*\big(h^k \cap[\Hilb_X^\beta]^\vir \big)\in H_{*}(\Pic_X^\beta)=\Lambda^{*} H^1(X,\Z).
\]
As conjectured in \cite{DKO} and proven in \cite{CK,DKO}, the Poincar\'e invariants are Seiberg-Witten invariants for all algebraic surfaces. For later use, we write each component of the Poincar\'e-Seiberg-Witten invariant as
\[\SW^k(\beta):=\mathsf{AJ}_*\big(h^k\cap [\Hilb_X^\beta]^\vir\big) \in H_{2(\vd_\beta -k)}(\Pic_X^\beta).
\]
Therefore the total Poincar\'e-Seiberg-Witten invariant is
\[\SW(\beta)=\sum_{k=0}^\infty \SW^k(\beta).\]

We say that a curve class $\beta$ is a Seiberg-Witten basic class if 
\[\big(\SW(\beta), \SW(K_X-\beta)\big)\neq (0,0).
\]
For surfaces with $p_g>0$, every Seiberg-Witten basic class $\beta$ satisfies $\vd_\beta=0$ by \cite[Proposition 4.20]{DKO}. When $\vd_\beta=0$, the higher Seiberg-Witten invariants $\SW^k(\beta)$ for $k>0$ vanish for degree reasons. In this case, $\SW(\beta)$ is simply an integer
$$\SW(\beta)=\SW^0(\beta)=\deg\big[\Hilb_X^\beta\big]^\vir.
$$

\subsection{Main results}\label{sMain}
\noindent
Let $X$ be a smooth projective surface over $\C$ and $\beta\in H^2(X,\Z)$ be a curve class. Let $\Quot_X(\C^N\!,\beta,n)$ be the Quot scheme parametrizing quotients
\[0 \rightarrow S\rightarrow \C^N\otimes \mathcal{O}_X \rightarrow Q\rightarrow 0
\]
where $Q$ is a torsion sheaf on $X$ with 
\[c_1(Q)=\beta, \ \ \chi(Q)=n.\]
By Subsection \ref{sdef}, this Quot scheme possesses a 2-term perfect obstruction theory, hence the virtual $\chi_{-y}$-genus 
\[\chi_{-y}^\vir\big(\Quot_X(\C^N\!,\beta,n)\big)\]
is defined. The following generating series of the invariants is our main object of study.
\begin{definition}
\[Z_{X, N,\,\beta}(q):=\sum_{n\in \Z} q^n \chi_{-y}^\vir\big(\Quot_X(\C^N\!,\beta,n)\big)\in \Z[y]((q)).\]
\end{definition}
This is a formal Laurent series in $q$ because the virtual dimension 
\[\vd\big(\Quot_X(\C^N\!,\beta,n)\big)=Nn+\beta^2
\]
becomes negative for sufficiently small $n$. The analogous generating series of the virtual Euler characteristics was defined in \cite[Definition 19]{OP} and it was conjectured to be rational in $q$.

We study the Quot scheme $\Quot_X(\C^N\!,\beta,n)$ via equivariant localization. Consider a diagonal $\C^*$-action on $\C^N$ with distinct weights \[w_1, \dots, w_N.\] This induces a $\C^*$-action on $\Quot_X(\C^N\!,\beta,n)$ and the obstruction theory possesses a natural $\C^*$-equivariant structure. Therefore the equivariant virtual $\chi_{-y}$-genus of the Quot scheme is defined. These equivariant structures will be fixed throughout the paper.

\subsubsection{Universality.}\label{subsection universality}
In Theorem \ref{tuniv} below, we express the equivariant Quot scheme invariants universally in terms of Seiberg-Witten invariants. Due to the Picard group factors, the Quot scheme invariants depend not only on intersection numbers but also on the ring structure of $H^*(X,\Z)$. To state the theorem, we introduce the following integral cohomology classes on the product of Picard groups. A similar definition for a single Picard factor appears in \cite{KT}. For any $\alpha\in H^{4-k}(X)$, define
\[[\alpha]_{s_1\leq \cdots \leq s_k}\in 
H^1(X)_{s_1}^\vee\otimes\cdots \otimes H^1(X)_{s_k}^\vee
\]
as a linear functional 
\[[\alpha]_{s_1\leq \cdots \leq s_k}
(\gamma_{s_1}\otimes \cdots\otimes \gamma_{s_k}):=
\int_X \alpha\cup\gamma_{s_1}\cup \cdots \cup \gamma_{s_k}.
\]
If there are equalities among $s_1\leq \cdots \leq s_k$, then we contract such parts to the wedge product. For example, if  $1\in H^0(X)$, then
\[[1]_{3=3\leq 6\leq7}\in \left(\Lambda^2H^1(X)_3^\vee\right)\otimes H^1(X)^\vee_6\otimes H^1(X)_7^\vee\]
where the indices $3,6$ and $7$ single out factors in the cohomology of the product of Picard groups
\[H^*(\Pic^{\beta_1}_X\times \cdots \times \Pic^{\beta_N}_X)=
\left(\Lambda^* H^1(X)_1^\vee\right)\otimes \cdots \otimes \left(\Lambda^* H^1(X)_N^\vee\right).
\]
\begin{theorem}\label{tuniv}
There exists a collection of the universal power series 
$$\mathsf{Q}_{\underline{k}, \underline{m}}\in \Q(y, e^{w_1}, \cdots, e^{w_N})[\![*]\!]
$$
for each $\underline{k}, \underline{m}\geq 0$ with the following properties.\footnote{Double bracket $[\![*]\!]$ is a place for the power series variables. Number of variables depends on $N$.}\footnote{Throughout the paper, we denote an $N$-tuple of objects (e.g. numbers, equivariant weights, cohomology classes) by vector notation (e.g.  $\underline{m}$, $\underline{w}$, $\underline{\beta}$).} Let $X$ be a smooth projective surface and $\beta$ be a curve class. Then equivariant virtual $\chi_{-y}$-genus of the Quot scheme is given by 
\begin{align*}
\chi_{-y}^\vir\big(\Quot_X(\C^N, \beta, n)\big)=\sum
\big(1-y\big)^{n+\sum_i\beta_i^2}\prod_{i\neq j} \left(\frac{1-ye^{w_i-w_j}}{1-e^{w_i-w_j}}\right)^{n_j+\beta_i.\beta_j}\int_{\prod\limits_i [\SW^{k_i}(\beta_i)]}
\mathsf{Q}_{\underline{k},\,\underline{m}}
\end{align*}
where the summation ranges over the choices of $\{\beta_i, n_i, k_i\}_{i=1}^N$ such that
$$\beta=\sum_{i=1}^N\beta_i, \quad n=\sum_{i=1}^N n_i, \quad m_i:=n_i+\frac{\beta_i(\beta_i+K_X)}{2}\geq 0, \quad k_i\geq 0.$$ 
\end{theorem}
\begin{remark} 
In Theorem \ref{tuniv}, $\mathsf{Q}_{\underline{k},\,\underline{m}}$ is a power series in the following variables: 
\begin{align*}
K_X^2\ \ , \ \
\chi(\OO_X)\ \ ,\ \
\big\{\beta_i^2\big\}_{1\leq i\leq N}\ \ ,\ \
\big\{\beta_i.K_X\big\}_{1\leq i\leq N}\ \ ,\ \
\big\{\beta_i.\beta_j\big\}_{1\leq i< j\leq N}\,,\\
\big\{[K_X]_{s\leq t}\big\}_{1\leq s\leq t\leq N}\ ,\
\big\{[\beta_i]_{s\leq t}\big\}_{1\leq i\leq N,\, 1\leq s\leq t\leq N}\ ,\
\big\{[1]_{s\leq t\leq p\leq q}\big\}_{1\leq s\leq t\leq p\leq q\leq N}\,.
\end{align*}
There are no convergence issues as $\mathsf{Q}_{\underline{k}, \underline{m}}$ is polynomial in the variables whose cohomological degree is $0$, i.e., the variables written in the first line above. 
\end{remark}

\subsubsection{Multiplicative structure.}
Despite the universal nature of the invariants, computations can be unwieldy, for instance, due to the higher Seiberg-Witten invariants appearing in Theorem \ref{tuniv}. Nonetheless, if all contributions come from effective decompositions $\beta=\sum \beta_i$ with $\vd(\beta_i)=0$ for all $i$, then the answer simplifies as in Theorem \ref{tmult} below. This will be effectively applied to give complete answers in several cases. 
\begin{definition}\label{dSW}
We say a curve class $\beta$ is of Seiberg-Witten length $N$ if for any effective decomposition
$\beta=\sum_{i=1}^N \beta_i$ with $\SW(\beta_i)\neq 0$ for all $i$, we have $\vd_{\beta_i}=0$ for all $i$. 
\end{definition}

\begin{remark}\label{rSW}
By \cite[Proposition 4.20]{DKO}, every curve class on a surface with $p_g>0$ is of Seiberg-Witten length $N$ for arbitrary $N$.
\end{remark}

\begin{theorem}\label{tmult}
Let $X$ be a smooth projective surface and $\beta$ be a curve class of Seiberg-Witten length $N$. There exist universal series 
\[\mathsf{U}_N, \ \ \mathsf{V}_{N,i}, \ \ \mathsf{W}_{N,i,j}\  \in\ 1+q\cdot\Q(y, e^{w_1}, \dots, e^{w_N})[[q]]
\]
such that the generating series of the equivariant virtual $\chi_{-y}$-genera takes the form
\[Z_{X,N,\,\beta}(q\,|\,\underline{w})=q^{-\beta.K_X}\!\!\!\!\!\!\!
\sum_{\substack{\beta=\sum \beta_i\\ s.t. \ \vd_{\beta_i}=0}}\!\!\!\!\!\!
\SW(\beta_1)\cdots \SW(\beta_N) \cdot \widehat{Z}_{X,N,\,\underline{\beta}}(q\,|\,\underline{w})
\]
where
\[\widehat{Z}_{X,N,\,\underline{\beta}}(q\,|\,\underline{w})=\mathsf{U}_N^{K_X^2}\ 
\prod_{i=1}^N\mathsf{V}_{N,i}^{\, \beta_i.K_X}\!\!\!\!\!\!
\prod_{1\leq i<j\leq N}\!\!\!\!\!\mathsf{W}_{N,i,j}^{\, \beta_i.\beta_j}.
\]
\end{theorem}
There is a formula \cite[Theorem A]{L} for the monopole contributions to the refined Vafa-Witten invariants which has a similar structure as Theorem \ref{tmult}. The difference is the conjectural modularity of the universal series on the Vafa-Witten side and the intersection numbers appearing as exponents.  

\subsubsection{Punctual quotients.}
There are a few other curve classes of Seiberg-Witten length $N$ without the assumption $p_g>0$. This includes the case $\beta=0$ for arbitrary surfaces. Since the only effective decomposition of  the zero curve class is $0=\sum \beta_i$ with $\beta_i=0$ for all $i$, the universal series $\mathsf{U}_N$ fully determines the answer, by Theorem \ref{tmult}. In Subsection \ref{s0}, we prove that $\mathsf{U}_N$ is rational in the $q$ variable
\[\mathsf{U}_N\in \Q(y, e^{w_1}, \dots, e^{w_N})(q)
\]
without poles at $e^{w_1}=\cdots=e^{w_N}=1$. Specialization to $w_1=\cdots=w_N=0$ 
\[\overline{\mathsf{U}}_N:=\mathsf{U}_N\Big\vert_{w_1=\cdots=w_N=0}
\]
gives the non-equivariant answer. In the following theorem, $\beta=0$ is omitted from the notation. 

\begin{theorem}\label{t0}
Let $X$ be a smooth projective surface. The generating series of the virtual $\chi_{-y}$-genera of\ \ $\Quot_X(\C^N, n)$ takes the form
\[Z_{X,\,N}(q)=\overline{\mathsf{U}}_N^{K_X^2}, \ \ \ \ \overline{\mathsf{U}}_N\in\Q(y)(q).\]
Here the rational function $\overline{\mathsf{U}}_N$ is given by
\[\overline{\mathsf{U}}_N=\frac{(1-q)(1-y^Nq)}{\big(1-(1+y)^Nq\big)^N}\cdot \prod_{i\neq j} \big(1-(1+y)t_i+yt_it_j\big)
\]
where $t_1, \dots, t_N$ are the distinct roots of $q=t^N$.
\end{theorem}
We can easily calculate $\mathsf{P}_N:=\prod\limits_{i\neq j}\big(1-(1+y)t_i+yt_it_j\big)$. The first few examples are 
\begin{align*}
\mathsf{P}_1&=1, \\
\mathsf{P}_2&=1-(1+4y+y^2)q+y^2q^2,\\
\mathsf{P}_3&=1-(2+9y+9y^2+2y^3)q+(1+9y+36y^2+58y^3+36y^4+9y^5+y^6)q^2\\
&\ \ \ \ \ \, -(2+9y+9y^2+2y^3)y^3q^3+y^6q^4.
\end{align*}
The explicit formula above implies the transformation rule
\begin{equation}\label{functional}
\mathsf{P}_N(q,y)=(y^Nq^2)^{N-1}\cdot \mathsf{P}_N(q^{-1},y^{-1}).
\end{equation}

\subsubsection{Relatively minimal elliptic surfaces.}
Another example where Theorem \ref{tmult} applies is the case of a relatively minimal elliptic surface $p:X\rightarrow C$. By the canonical bundle formula, $K_X$ is a rational multiple of the fiber class $F$. Consider a curve class $\beta$ in the fibers, so that $\beta.F=0$. For any effective decomposition $\beta=\sum \beta_i$, we have $\beta_i.F=0$. By Zariski's lemma \cite{BHPV}, we have 
\[ \vd_{\beta_i}=\frac{\beta_i(\beta_i-K_X)}{2}=\frac{\beta_i^2}{2}\leq 0
\]
with the equality if and only if $\beta_i$ is a rational multiple of $F$. Therefore, $\beta$ is of Seiberg-Witten length $N$ for arbitrary $N$. Based on this observation, Theorem \ref{tmult} directly implies the following:
\begin{theorem}\label{tell}
Let $X$ be a relatively minimal elliptic surface over a smooth curve. Let $\beta$ be an effective curve class supported on the fibers. Then 
\[Z_{X,N,\,\beta}(q)=\!\!\!\!\!\!\!\!\!\!\!\!\!\!\!\!\!\!
\sum_{\substack{\beta=\sum \beta_i, \\ \beta_i \textnormal{ rational multiple of $F$}}}\!\!\!\!\!\!\!\!\!\!\!\!\!\!\!
\SW(\beta_1)\cdots \SW(\beta_N).\]
In particular, the Quot scheme invariants vanish for curve classes supported on the fibers which are not rational multiples of $F$. 
\end{theorem}
The Seiberg-Witten invariants appearing in Theorem \ref{tell} are computed in \cite[Proposition 5.8]{DKO}. Let $F_1, \dots, F_r$ be the multiple fibers with multiplicities $m_1, \dots, m_r$. Then for each $\beta_i$ a rational multiple of $F$,
\[\SW(\beta_i)=\sum_{\substack{\beta_i=dF+\sum_{j}a_jF_j\\ s.t. \  0\leq a_j< m_j}}(-1)^d {2g-2+\chi(\mathcal{O}_X)\choose d}
\]
where $g$ is the genus of the base curve $C$. 
\subsubsection{Minimal surfaces of general type with $p_g>0$.}
Let $X$ be a minimal surface of general type with $p_g>0$. In this case, \cite{DKO} proved that the Seiberg-Witten invariants vanish unless the curve class is $0$ or $K_X$. This observation simplifies the formula of Theorem \ref{tmult}. To state Theorems \ref{tg} and \ref{tbl} in a uniform way, we introduce certain functions which appear in both. Let $J$ be a subset of $[N]:=\{1, \dots, N\}$ of size $s$. We define
\begin{align*}
\mathsf{A}_{J}(\{t_j\}_{j\in J}):=\frac{(-1)^{s(N+1)}}{N^s \cdot q^s} \times
\prod_{j} \frac{t_j (1-(1+y)t_j)^N}{(1-t_j)(1-yt_j)} \times
\prod_{j_1\neq j_2} \frac{t_{j_2}-t_{j_1}}{1-(1+y)t_{j_1}+yt_{j_1}t_{j_2}}
\end{align*}
where the products are taken over indices in $J$ and $t_1, \dots, t_N$ are the $N$ distinct roots of the equation $q=t^N$. We omit $N$ in the notation $\mathsf{A}_J$ because it will be clear from the context.
\begin{theorem}\label{tg}
Let $X$ be a minimal surface of general type with $p_g>0$. Then, for any $0\leq \ell\leq N$,
\[Z_{X,N,\, \ell K_X}(q)=q^{-\ell K_X^2}\cdot \SW(K_X)^{\ell}\cdot \mathsf{G}_{N,\,\ell,\, g}(q)\, ,\ \ \ \ \mathsf{G}_{N,\, \ell,\, g}\in \Q(y)(q)
\]
where $g=K_X^2+1$ is the arithmetic genus of a canonical divisor. 
Here the rational function $\mathsf{G}_{N,\,\ell,\, g}$ is given by 
\[\mathsf{G}_{N,\,\ell,\, g}=\sum_{\substack{J\subseteq[N]\\ |J|=N-\ell}} (\mathsf{A}_{J})^{1-g}.\]
Furthermore, the Quot scheme invariants of any other curve classes vanish. 
\end{theorem} 
For minimal surfaces of general type with $p_g>0$, it is proven in \cite{CK} that 
\[\SW(K_X)=(-1)^{\chi(\mathcal{O}_X)}.\] 
Substituting this, Theorem \ref{tg} refines \cite[Theorem 23]{OP} to virtual $\chi_{-y}$-genera and also removes the technical condition that $X$ be simply connected with a smooth canonical curve.  
\subsubsection{Blow up formula.}
By Lemma \ref{lbl} in Subsection \ref{sbl}, curve classes of Seiberg-Witten length $N$ stay of Seiberg-Witten length $N$ under blow ups. Using this and Theorem \ref{tmult}, we obtain the following blow up formula.
\begin{theorem}\label{tbl}
Let $X$ be a smooth projective surface and $\beta$ be a curve class of Seiberg-Witten length $N$. Let $\pi:\widetilde{X}\rightarrow X$ be a blow up of a point with the exceptional divisor $E$. Denote the pull back class by $\widetilde{\beta}:=\pi^*\beta$. Then, for any $0\leq \ell\leq N$,
\[Z_{\widetilde{X},\,N,\,\widetilde{\beta}+\ell E}(q)
=\left(q^\ell\cdot \mathsf{Bl}_{N,\,\ell}(q)\right) Z_{X,\,N,\,\beta}(q)\, ,\quad \mathsf{Bl}_{N,\,\ell}\in \Q(y)(q).
\]
Here the rational function $\mathsf{Bl}_{N,\,\ell}$ is given by 
\[\mathsf{Bl}_{N,\,\ell}=\sum\limits_{\substack{J\subseteq[N]\\ |J|=N-\ell}} \mathsf{A}_{J},\]
independently on the surface. Furthermore, the Quot scheme invariants vanish for other values of $\ell$. 
\end{theorem}
\begin{remark}
The formula in Theorem \ref{tbl} shares two important features with the blow up formula for invariants of moduli of sheaves in \cite{GK,LQ1,LQ2}. First, it is multiplicative with respect to the generating series of invariants. Second, the formula is independent of the surface and the first Chern class. 
\end{remark}
For any given small values of $N, \ell, g$, it is easy to calculate $\mathsf{Gl}_{N,\,\ell,\,g}$ and $\mathsf{Bl}_{N,\,\ell}$ by computer. Furthermore, we can also prove that
\[\mathsf{Bl}_{N,\,N}=1, \quad \mathsf{Bl}_{N,\,N-1}=\frac{1}{1-y}\cdot \frac{(1-y^N)-(1-y^{2N})q}{(1-q)(1-y^Nq)}.\]

\subsubsection{Surfaces with $p_g>0$ and rationality of generating series.}

For surfaces with $p_g>0$, the blow up formula of Theorem \ref{tbl} reduces the study of the Quot scheme invariants to the case of minimal surfaces. By the classification of surfaces in terms of the Kodaira dimension, minimal surfaces with $p_g>0$ are of the form:
\begin{enumerate}
\item[(i)] \kod=0: $K3$ or abelian surfaces,
\item[(ii)] \kod=1: minimal elliptic surfaces with $p_g>0,$
\item[(iii)] \kod=2: minimal surfaces of general type with $p_g>0$.
\end{enumerate} 
Combined with an understanding of the Seiberg-Witten basic classes in each case, Theorems \ref{t0}, \ref{tell}, \ref{tg}, \ref{tbl} prove the following result:
\begin{theorem}\label{trat}
For all surfaces $X$ with $p_g>0$ and all curve classes $\beta$, the series $Z_{X,N,\,\beta}(q)$ is a rational function in $q$.
\end{theorem}
This theorem resolves a conjecture of \cite{OP} regarding the rationality of the generating series of the virtual $\chi_{-y}$-genera, for surfaces with $p_g>0$.

The rationality question for surfaces with $p_g=0$ is more difficult. Recently \cite{JOP} established rationality (for virtual Euler characteristics) when $N=1$ for all simply connected surfaces. There, the same question was studied in the more general setting of series with descendant insertions.

\subsubsection{Reduced invariants of $K3$ surfaces.}
Let $X$ be a $K3$ surface and $\beta$ be an irreducible curve class which is big and nef. In this case, the virtual fundamental class constructed from the usual obstruction theory vanishes. To define nontrivial invariants, we use the reduced obstruction theory. When $N=1$, \cite{OP} matched the reduced virtual Euler characteristic of the Quot scheme to the topological Euler characteristic of the relative Hilbert scheme of points on the universal curve. A similar argument is valid for the reduced virtual $\chi_{-y}$-genus. 
\begin{theorem}\label{tK3}
Let $X$ be a $K3$ surface and $\beta$ be an irreducible curve class which is big and nef, of arithmetic genus $g$ given by $2g-2=\beta^2$. Let $\mathcal{C}_g\rightarrow \PP$ be the universal curve over the linear series $|\beta|$. Then,
\[\chi_{-y}^\red\left(\Quot_X(\C^1\!, \beta,n)\right)=\chi_{-y}(\mathcal{C}_g^{[m]})\]
where $m=n+(g-1)$ and $\mathcal{C}_g^{[m]}$ is the relative Hilbert scheme of $m$ points on the universal curve. 
\end{theorem}
The relative Hilbert schemes of points are smooth under the above assumption and their Hodge polynomials are computed in \cite{KY}. Specializing to the shifted $\chi_{-y}$-genus, one obtains
\begin{equation*}
\sum_{g\geq 0}\sum_{m\geq 0} \overline{\chi}_{-y}(\mathcal{C}_g^{[m]})t^{m+1-g}q^{g-1}
=\frac{y^{-1/2}-y^{1/2}}{\Delta(q)}\frac{\theta'(1)^3}{\theta(y^{1/2}/t)\theta(ty^{1/2})\theta(y)}.
\end{equation*}
See Subsection \ref{proof of K3} for the definition of the modular forms in the formula, and \cite{GS} for details. Here the shifted $\chi_{-y}$-genus is 
\[\overline{\chi}_{-y}(M):={\chi}_{-y}(M) \cdot y^{-\dim (M) /2}
\]
with the obvious analogue for the virtual $\chi_{-y}$-genus.
Define the generating series of the shifted reduced invariants as
\[\overline{Z}^{\red}_{X,1,\, \beta}(t)
:=\sum_{n\in \Z} \overline{\chi}_{-y}^\red\left(\Quot_X(\C^1\!, \beta,n)\right) t^n.\]
Combining Theorem \ref{tK3} with the Kawai-Yoshioka formula directly yields the following:
\begin{corollary}\label{cor}
Let $X$ be a $K3$ surface and $\beta$ be a primitive curve class which is big and nef, of arithmetic genus $g$ given by $2g-2=\beta^2$. Then $\overline{Z}^{\red}_{X,1,\, \beta}(t)$\ \!
is the coefficient of $q^g$ in the expression
\begin{multline*}
\frac{1}{(t+\frac{1}{t}-\sqrt{y}-\frac{1}{\sqrt{y}})}\times\\
\prod\limits_{n>0} \frac{1}{(1-q^n)^{18}(1-q^ny)(1-q^n/y)(1-q^n\sqrt{y}/t)(1-q^nt/\sqrt{y})(1-q^nt\sqrt{y})(1-q^n/t\sqrt{y})
}.
\end{multline*}
In particular, $Z^{\red}_{X,1,\, \beta}(t)$ is a rational function in the $t$ variable. 
\end{corollary}
Formulas for the reduced invariants for higher $N$ are not available yet. Note however that complete expressions for the higher rank stable pair invariants on $K3$ surfaces with primitive curve classes were obtained in \cite{BJ}. We leave the study of the higher $N$ case and the comparison to its stable pair analogue for future work. 

Finally, note that the case $\beta=0$ on a $K3$ surface is not covered by Corollary \ref{cor}. Nevertheless, we can compute the reduced invariants for dimension 0 quotients by Theorem \ref{tK3} via deformation to an elliptic $K3$ surface.  
\begin{theorem} \label{tred}
Let X be a $K3$ surface. The generating series of reduced $\chi_{-y}$-genera of $X^{[n]}$ takes the form
\[\sum_{n=1}^\infty t^n\, \chi_{-y}^\red (X^{[n]})=\frac{t\,(2+20y+2y^2)}{(1-yt)(1-t)}.
\]
\end{theorem}

\subsection{Plan of the paper}
\noindent We start by studying the Quot scheme and its virtual fundamental class for $N=1$ in Subsection \ref{suniversality}. We then apply virtual localization and the recursive argument of \cite{EGL} to prove the universality of the Quot scheme invariants in Theorem \ref{tuniv}. For curve classes of Seiberg-Witten length $N$, multiplicativity of the answer is shown. In Section \ref{sapp}, the universal series in the multiplicative formula of Theorem \ref{tmult} are studied by using special geometries. We use $K3$ surfaces to complete the proof of Theorem \ref{tmult}. Next we prove Theorem \ref{t0} by using blow ups of $K3$ surfaces. Elliptic surfaces are studied next and Theorem \ref{tell} is established. Theorem \ref{tg} and Theorem \ref{tbl} are proven in Subsection \ref{sgl} and \ref{sbl}, respectively. We then use Theorems \ref{t0}, \ref{tell}, \ref{tg}, \ref{tbl} to prove Theorem \ref{trat} in Subsection \ref{srat}. In Section \ref{sK3}, we study the reduced invariants of $K3$ surfaces and prove Theorem \ref{tK3} and Theorem \ref{tred}.

\subsection{acknowledgements}
The author would like to express special thanks to Dragos Oprea for suggesting this topic and numerous discussions. The author also thank the anonymous referees for various suggestions which improved the manuscript. This work is based on the research which will be part of the author's thesis at UCSD.

\section{The structure of the Quot scheme invariants}\label{suniversality}
\subsection{The virtual fundamental class of the Hilbert scheme} \label{sN=1}
\noindent We begin by studying the virtual fundamental classes of Quot schemes with $N=1$ using results of \cite{GSY,GT}. This subsection will be applied to the case of general $N$ in Subsection \ref{suniv}. 

Let $X$ be a smooth projective surface and $\beta$ be a curve class. When $N=1$, the Quot scheme $\Quot_X(\C^1\!, \beta,n)$ is simply the Hilbert scheme of 1-dimensional subschemes on $X$. The lemma below allows us to separate points and divisors. The proof of the lemma was given in \cite{F}, but we review some of the details for later use. 

\begin{lemma}\label{lstructure}
Let $X$ be a smooth projective surface and $\beta$ be a curve class. Then 
\[\Quot_X(\C^1\!, \beta,n)\simeq X^{[m]}\times \textnormal{Hilb} ^\beta_X\, ,
\]
where 
\[m=n+\frac{\beta(\beta+K_X)}{2}.\]
\end{lemma}
\begin{proof}
We describe the isomorphism of the lemma on the level of closed points. Take $(Z,D)\in X^{[m]}\times \Hilb^\beta_X$. Define $S=I_Z(-D)$ as the ideal sheaf of the 1-dimensional subscheme given by a composition of two injections
\[S=I_Z(-D)\hookrightarrow\OO_X(-D)\hookrightarrow \OO_X.
\]
Then the corresponding quotient
\[0\rightarrow S\rightarrow \OO_X\rightarrow Q\rightarrow 0
\]
satisfies $c_1(Q)=\beta$ and $\chi(Q)=m-\frac{\beta(\beta+K_X)}{2}$. This provides the one direction of the correspondence. 

For the other direction, take 
\begin{equation}\label{11}
0 \rightarrow S\rightarrow \OO_X\rightarrow Q\rightarrow 0
\end{equation}
with $c_1(Q)=\beta, \ \chi(Q)=n$. By taking double duals, we obtain the following diagram:\\
\begin{center}
\begin{tikzcd}%
&&&0\arrow[d]&\\
& 0 \arrow[r] \arrow[d] & 0 \arrow[r]  \arrow[d] & Q' \arrow[dash, dotted, r] \arrow[d]  &  \ \\ 
0 \arrow[r] & S \arrow[r] \arrow[d] & \OO_X \arrow[r]  \arrow[d] & Q \arrow[r] \arrow[d]  & 0 \\  
0 \arrow[r] & S^{\vee \vee} \arrow[r] \arrow[d] & \OO_X^{\vee \vee} \arrow[r]  \arrow[d] & Q'' \arrow[r] \arrow[d]  & 0 \\  
\ \arrow[dotted, r] & Q' \arrow[r] \arrow[d]  & 0 \arrow[r]  & 0 \rlap{ .} & \\
&0&&&
\end{tikzcd}%
\end{center}
A torsion free sheaf on a smooth surface injects into its double dual with a zero dimensional cokernel.  Therefore $Q'$ is a zero dimensional sheaf. Also, $S^{\vee \vee}$ is a reflexive sheaf on a smooth surface hence locally free. Therefore we can regard
\[0 \rightarrow S^{\vee \vee} \rightarrow \OO_X^{\vee \vee} \rightarrow Q''\rightarrow0\]
as the ideal exact sequence of an effective divisor $D$ and identify $S^{\vee \vee}=\mathcal{O}(-D)$. 
Twisting the first column by $\OO_X(D)$, we also obtain
\[0 \rightarrow S(D)\rightarrow \OO_X\rightarrow Q'(D)\rightarrow0\]
which can be thought of as the ideal exact sequence of a zero dimensional subscheme $Z$ and identify $Q'(D)=\mathcal{O}_Z$. Therefore, the last column is identified with the exact sequence
\begin{align*}
0\rightarrow \OO_Z(-D)\rightarrow Q\rightarrow \OO_D\rightarrow0.
\end{align*}
Taking holomorphic Euler characteristics, we obtain
\begin{align*}
m&:=\text{length}(Z)=\chi(\OO_Z(-D))=\chi(Q)-\chi(\OO_D)\\&=n+\frac{\beta(\beta+K_X)}{2}
\end{align*}
by the Riemann-Roch theorem. The assignment from (\ref{11}) to $(Z, D)$ provides the other direction of the correspondence. 

It is clear that these constructions are inverse to each other. 
The same argument in families completes the proof.
\end{proof}
We next study the deformation theory of the Quot scheme above. In general, a natural 2-term perfect obstruction theory for the nested Hilbert scheme was constructed in \cite{GSY}. As a particular case, consider the nested Hilbert scheme $X_\beta^{[m_1, m_2]}$ parametrizing zero dimensional subschemes $(Z_1, Z_2)$ and a divisor $D$ such that
\[ \textnormal{length}(Z_i)=m_i, \quad c_1(\OO(D))=\beta,\quad  I_{Z_1}(-D)\subseteq I_{Z_2}.
\]
The construction can be easily understood in three special cases. Specifically, the following identifications have been proved in \cite[Proposition 3.1]{GSY}:
\[[X^{[m,m]}_{\beta=0}]^\vir=[X^{[m]}], \quad 
[X^{[0,0]}_\beta]^\vir=[\Hilb_X^\beta]^\vir, \quad
[X^{[0,m]}_\beta]^\vir=[P_n(X,\beta)]^\vir.
\]
Here $P_n(X,\beta)$ is the moduli space of stable pairs. By Lemma 1, we also have the identification 
\[\Quot_X(\C^1\!, \beta,n)=X_{\beta}^{[m, 0]}.\]
The Proposition below matches the virtual fundamental classes.
\begin{proposition}
Through the identification $\Quot_X(\C^1\!, \beta,n)=X^{[m,0]}_\beta$, the virtual fundamental class of the Quot scheme on the left and the virtual fundamental class of the nested Hilbert scheme on the right agree:
\[[\Quot_X(\C^1\!, \beta,n)]^\vir=[X^{[m,0]}_\beta]^\vir.
\]
\end{proposition}
\begin{proof}
For a projective scheme with 2-term perfect obstruction theory, the virtual fundamental class depends only on the virtual tangent bundle in the $K$-theory and the underlying scheme structure by \cite[Theorem 4.6]{S}. Therefore it suffices to compare the virtual tangent bundle of $X^{[m]}\times \Hilb_X^\beta$ as a Quot scheme and as a nested Hilbert scheme. By Lemma \ref{lstructure}, the universal quotient over 
\[ \Quot_X(\C^1\!, \beta,n)\times X=
\left(X^{[m]}\times \Hilb_X^\beta\right)\times X
\]
is given by 
\[ 0\rightarrow I_\ZZ(-\DD)\rightarrow\OO_{X^{[m]}\times \Hilb_X^\beta\times X}\rightarrow \mathcal{Q}\rightarrow 0,
\]
where $\ZZ$ and $\DD$ denote the universal subscheme and the universal divisor.
Using this, we compute the virtual tangent bundle of $X^{[m]}\times \Hilb_X^\beta$ as a Quot scheme
\begin{align*}
T^\vir_{\Quot}&=R\HHom_{\pi}(I_\ZZ(-\DD), \mathcal{Q})\\
&=R\HHom_{\pi}(I_\ZZ(-\DD),\OO-I_\ZZ(-\DD))\\
&=-R\HHom_{\pi}(I_\ZZ,I_\ZZ) +R\HHom_{\pi}(I_\ZZ,\OO(\DD))\\
&=R\HHom_\pi(I_\ZZ,I_\ZZ)_0[1]-R\HHom_\pi(\OO,\OO)+R\HHom_\pi(\OO-\OO_\ZZ, \OO+\OO_\DD(\DD))\\
&=T_{X^{[m]}}+R\pi_*(\OO_\DD(\DD))-R\HHom_\pi(\OO_\ZZ,\OO(\DD))\\
&=T_{X^{[m]}}+T^\vir_{\Hilb_X^\beta}-R\HHom_\pi(\OO_\ZZ,\OO(\DD))
\end{align*}
where the subscript $0$ on the fourth line denotes the trace free part. 
We have suppressed various pull backs and denoted the projection map from $X$ to a point by $\pi$. 

By \cite[Proposition 2.2]{GSY}, the virtual tangent bundle of $X^{[m]}\times \Hilb_X^\beta$ as a nested Hilbert scheme is
\[T^\vir_{X^{[m,0]}_\beta}=
R\HHom_{\pi}(I_\ZZ,\OO(\DD))-\Big[R\HHom_{\pi}(I_\ZZ,I_\ZZ)+R\HHom_{\pi}(\OO,\OO)\Big]_0
\]
where the bracket with subscript $0$ denotes again the trace free part. This removes one copy of $R\HHom_\pi(\OO,\OO)$ in $K$-theory, hence
\[T^\vir_{X^{[m,0]}_\beta}=-R\HHom_{\pi}(I_\ZZ,I_\ZZ) +R\HHom_{\pi}(I_\ZZ,\OO(\DD)).\]
This agrees with the third line of the computation of $T^\vir_\Quot$. This competes the proof.
\end{proof}
In the above proof, we established that 
\begin{equation}\label{vtan}
T^\vir_{X^{[m]}\times \Hilb_X^\beta}=T_{X^{[m]}}+T^\vir_{\Hilb_X^\beta}-R\HHom_\pi(\OO_\ZZ,\OO(\DD)).
\end{equation}
Note that 
\[R\HHom_\pi(\OO_\ZZ,\OO(\DD))=R\HHom_\pi(\OO(\DD),\OO_\ZZ(K_X))^\vee=\Big(\pi_*\OO_\ZZ(K_X-\DD)\Big)^\vee
\]
is an actual, not virtual,  vector bundle of rank $m$ by Serre duality and the vanishing of the higher pushforwards. If $\Hilb_X^\beta$ were smooth, then this would easily imply 
\[ [X^{[m]}\times \Hilb_X^\beta]^\vir=e\big(R\HHom_\pi(\OO_\ZZ,\OO(\DD))\big) \cap
\Big([X^{[m]}]\times [\Hilb_X^\beta]^\vir\Big)
\]
where $e$ denotes the Euler class. 
When $\Hilb_X^\beta$ is singular, however, this is a special case of the main result in \cite{GT}.
\begin{proposition}\label{pvir}
The virtual fundamental class of $X^{[m]}\times \Hilb_X^\beta$ as a Quot scheme, or equivalently as a nested Hilbert scheme, is 
\[[X^{[m]}\times \Hilb_X^\beta]^\vir=e\big(R\HHom_\pi(\OO_\ZZ,\OO(\DD))\big) \cap
\Big([X^{[m]}]\times [\Hilb_X^\beta]^\vir\Big).
\]
\end{proposition}
\begin{proof}
We use \cite[Theorem 5]{GT}. By the Theorem, the virtual fundamental class on the left is identified with the Euler class of the Carlsson-Okounkov $K$-theory class. In our case, the relevant Carlsson-Okounkov $K$-theory class is
\begin{align*}
\mathsf{CO}^{[m,0]}_\beta&:=R\pi_*\OO(\DD)-R\HHom_\pi(I_\ZZ, \OO(\DD))\\
&=R\pi_*\OO(\DD)-R\HHom_\pi(\OO-\OO_\ZZ, \OO(\DD))\\
&=R\HHom_\pi(\mathcal{O}_\ZZ,\OO(\DD)).
\end{align*}
This directly implies the Proposition. 
\end{proof}

\subsection{Universality of the Quot scheme invariants}\label{suniv}
\noindent In this subsection, we prove Theorem \ref{tuniv}. Specifically, we study the equivariant Quot scheme invariants for arbitrary $N$ and express them universally in terms of higher Seiberg-Witten invariants. The proof of Theorem \ref{tuniv} breaks into three steps. First, we apply the virtual localization theorem of \cite{GP} and express the contributions from each fixed locus as an integral. Second, we pushforward the resulting integral under the Abel-Jacobi map. The Seiberg-Witten invariants naturally arise in this step. Third, we apply a generalization of the recursive argument in \cite{EGL} to describe tautological integrals over the Hilbert schemes of points in terms of universal series. 
\subsubsection{Step 1.}
Consider a diagonal $\C^*$-action on $\C^N$ with distinct weights $w_1, \dots, w_N$. We use the bracket to denote the equivariant weights. For example, the previous diagonal $\C^*$-action on $\C^N$ is denoted by 
\[\C[w_1]+\cdots +\C[w_N].
\]
This induces an $\C^*$-action on $\Quot_X(\C^N\!,\beta,n)$ via the associated $\C^*$-action on the middle term of the sequences
\[0\rightarrow S\rightarrow \C^N\otimes \mathcal{O}_X\rightarrow Q\rightarrow0.
\]
By the weight decomposition of a subsheaf, the fixed point loci parametrize short exact sequence
\[0\rightarrow S_1 [w_1] + \cdots +S_N [w_N]
\rightarrow \mathcal{O}[w_1] + \cdots +\mathcal{O}[w_N]
\rightarrow Q_1 [w_1]+\cdots +Q_N [w_N]
\rightarrow0
\]
for each decomposition of the discrete data 
\[n=\sum_{i=1}^N n_i,\quad \beta=\sum_{i=1}^N\beta_i\ ,\quad  \textnormal{ where } \quad n_i=\chi(Q_i)
,\ \  \beta_i=c_1(Q_i).
\]
Therefore, 
\[\left(\Quot_X(\C^N\!,\beta,n)\right)^{\C^*}
=\!\!\!\!\mathop{\bigsqcup}_{\substack{n=n_1+\cdots +n_N\\ \beta=\beta_1+\cdots+ \beta_N}} 
\!\!\!\!\Quot_X(\C^1\!, \beta_1,n_1)\times \cdots \times \Quot_X(\C^1\!, \beta_N,n_N).
\]

Now we study the contributions of each fixed locus 
\[F=\Quot_X(\C^1\!, \beta_1,n_1)\times \cdots \times \Quot_X(\C^1\!, \beta_N,n_N)\]
corresponding to the decomposition $\underline{n}=(n_1, \dots, n_N)$ and $\underline{\beta}=(\beta_1, \dots, \beta_N)$.  Encode the lengths of the zero dimensional parts in a vector
\[\underline{m}=(m_1, \dots, m_N)\quad \text{with}\quad m_i=n_i+\frac{\beta_i(\beta_i+K_X)}{2}.\] 
By Lemma \ref{lstructure}, we can identify
\[F=\Big(X^{[m_1]}\times\cdots \times  X^{[m_N]}\Big)
\times \Big(\Hilb_X^{\beta_1}\times \cdots \times \Hilb_X^{\beta_N}\Big)
\]
whose universal quotient is given by
\[0\rightarrow \sum_{i=1}^N I_{\ZZ_i}(-\DD_i)[w_i]\rightarrow\sum_{i=1}^N \OO_{F\times X}[w_i]\rightarrow\sum_{i=1}^N \mathcal{Q}_i[w_i]\rightarrow0.
\]
Using this, we can compute the virtual tangent bundle of the Quot scheme restricted to the fixed locus $F$
\[T^\vir_\Quot\Big\vert_F=\sum_{i,\, j=1}^NR\HHom_\pi(I_{\ZZ_i}(-\DD_i), \QQ_j)[w_j-w_i].
\]
By taking the fixed part and moving part, we obtain the virtual tangent bundle and virtual normal bundle of $F$, respectively:
\begin{align}
T^\vir_F&=\sum_{i=1}^N R\HHom_\pi(I_{\ZZ_i}(-\DD_i), \QQ_i), \nonumber\\
N^\vir_F&=\sum_{i\neq j} R\HHom_\pi(I_{\ZZ_i}(-\DD_i), \QQ_j)[w_j-w_i]=:\sum_{i\neq j} N_{ij}[w_j-w_i]. \label{xx}
\end{align}
The first equation above shows that the induced obstruction theory of $F$ splits into the obstruction  theories of the factors. The splitting of the obstruction theory, together with Proposition \ref{pvir}, gives
\begin{align*}
[F]^\vir&=[\Quot_X(\C^1\!, \beta_1, n_1)]^\vir\times \cdots \times [\Quot_X(\C^1\!, \beta_N, n_N)]^\vir \\
&=\prod_{i=1}^N e(\Obs_i)\cap \Big([X^{[m_1]}]\times\cdots \times [X^{[m_N]}]\times  [\Hilb_X^{\beta_1}]^\vir\times \cdots \times [\Hilb_X^{\beta_N}]^\vir\Big),
\end{align*}
where
\begin{equation}\label{x}
\Obs_i:=R\HHom_\pi(\OO_{\ZZ_i},\OO(\DD_i))
\end{equation}
are vector bundles of rank $m_i$. Write
\[[X^{[\underline{m}]}]:=[X^{[m_1]}]\times\cdots \times [X^{[m_N]}], \quad 
[\Hilb_X^{\underline{\beta}}]^\vir:=[\Hilb_X^{\beta_1}]^\vir\times \cdots \times [\Hilb_X^{\beta_N}]^\vir.
\]
By the virtual localization theorem \cite{GP}, we have
\begin{align*}
\chi_{-y}^\vir\big(\Quot_X(\C^N\!,\beta,n)\big)=\sum_F\int_{[F]^\vir} \frac{\XX_{-y}(T^\vir_\Quot\big\vert_F)}{e(N^\vir_F)}.
\end{align*}
Using the above formula of $[F]^\vir$ and the decomposition 
\[T^\vir_\Quot\Big\vert_F=T^\vir_F+N^\vir_F=\sum_i\Big(T_{X^{[m_i]}}+T^\vir_{\Hilb_X^{\beta_i}}-\Obs_i\Big)+N^\vir_F,
\]
the contribution of $F$ to the virtual $\chi_{-y}$-genus is
\begin{align} \label{step1}
\int_{[X^{[\underline{m}]}]\times[\Hilb_X^{\underline{\beta}}]^\vir}
\frac{e}{\XX_{-y}}\left(\sum_i \Obs_i\right)\  
\frac{\XX_{-y}}{e}\left(N^\vir_F\right)\ 
\XX_{-y}\left(\sum_i T_{X^{[m_i]}}+T^\vir_{\Hilb_X^{\beta_i}}\right).
\end{align}
\subsubsection{Step 2.}\label{sssss}
We begin with an observation that clarifies the dependence of the integrand in (\ref{step1}) on the ranks of $K$-theory classes. Let $E$ be a vector bundle of rank $r$ with Chern roots $\alpha_1, \dots, \alpha_r$. Write
\begin{align*}
\XX_{-y}(E)
&=\prod_{i=1}^r \frac{\alpha_i(1-ye^{-\alpha_i})}{1-e^{-\alpha_i}}
=(1-y)^r\cdot \prod_{i=1}^r \frac{\alpha_i\left(1+\frac{y}{1-y}(1-e^{-\alpha_i})\right)}{1-e^{-\alpha_i}}.
\end{align*}
From here, it can be seen that the dependence of $\XX_{-y}(E)$ on the rank $r$ is present in the prefactor $(1-y)^r$. The remaining terms give power series, independent of $r$, in the Chern classes of $E$ with coefficients in $\Q(y)$. This is because the dependence on the rank only comes from
\[ r=\sum_{i=1}^r(\alpha_i)^0,
\]
whereas the remaining terms involve only positive powers of Chern roots $\alpha_i$. There are no issues about convergence of the power series because as long as we fix a dimension, all but finitely many terms vanish. Similarly, for each $w\neq 0$, we have 
\begin{align*}
\frac{\XX_{-y}}{e}\big(E[w]\big)&=\prod_{i=1}^r\frac{1-ye^{-w-\alpha_i}}{1-e^{-w-\alpha_i}}
=\left(\frac{1-ye^{-w}}{1-e^{-w}}\right)^r\cdot
\prod_{i=1}^r\frac{1+\frac{ye^{-w}}{1-ye^{-w}}(1-e^{-\alpha_i})}
{1+\frac{e^{-w}}{1-e^{-w}}(1-e^{-\alpha_i})}\, ,
\end{align*}
where the  dependence on the rank is expressed in the first factor and the other factor is a power series in the Chern classes of $E$ with coefficients in $\Q(y,e^{w})$. 

By these observations, the integral (\ref{step1}) becomes
\begin{align*}
(1-y)^{\sum_i m_i+\vd_{\beta_i}} \cdot 
\prod_{i\neq j} \left(\frac{1-ye^{w_i-w_j}}{1-e^{w_i-w_j}}
\right)^{\rk (N_{ij})}\cdot \int_{[X^{[\underline{m}]}]\times[\Hilb_X^{\underline{\beta}}]^\vir}
\mathsf{A}_{\underline{m}} \, ,
\end{align*}
where $\mathsf{A}_{\underline{m}}$ is a power series of the Chern classes of 
\begin{equation}\label{Kclasses step2'}
\Obs_i, \quad N_{ij}, \quad T_{X^{[m_i]}}, \quad T^\vir_{\Hilb_X^{\beta_i}} 
\end{equation}
with coefficients in $\Q(y, e^{w_1}, \dots, e^{w_N})$. The universal power series $\mathsf{A}_{\underline{m}}$ depends on $\underline{m}$ because of  $e(\Obs_i)=c_{m_i}(\Obs_i)$. It is straightforward to check
\begin{gather*}
\sum_i m_i+\vd_{\beta_i}=n+\sum_i \beta_i^2,\\
\rk(N_{ij})=\chi\big(I_{Z_i}(-D_i), \OO-I_{Z_j}(-D_j)\big)=n_j+\beta_i.\beta_j.
\end{gather*}

Now we pushforward the integral 
\begin{equation}\label{step2'}
\int_{[X^{[\underline{m}]}]\times[\Hilb_X^{\underline{\beta}}]^\vir}
\mathsf{A}_{\underline{m}}
\end{equation}
along the product of Abel-Jacobi maps
\[ \mathsf{AJ}_{\beta_i}:\Hilb_X^{\beta_i}\rightarrow\Pic_X^{\beta_i}.
\]
Recall the geometric setup for the Seiberg-Witten invariants in Subsection \ref{sSW}. We fix the point $p\in X$ once and for all. Define the tautological line bundle over $\Hilb_X^{\beta_i}$ by
\[\LL_i:=\OO(\DD_i)|_{\Hilb_X^{\beta_i}\times \{p\}}
\]
and denote $h_i:=c_1(\LL_i)$. By the seesaw theorem, we have 
\[\OO(\DD_i)=\mathcal{P}_i\otimes \LL_i
\]
on $\Hilb_X^{\beta_i}\times X$, where $\mathcal{P}_i$ is the Poincar\'e line bundle on $\Pic_X^{\beta_i}\times X$ normalized at $p$. Using this relation and equations (\ref{xxx}), (\ref{xx}), (\ref{x}), we can express each $K$-theory class in (\ref{Kclasses step2'}) as a combination of $K$-theory classes pulled back from the Abel-Jacobi maps tensored by line bundles $\LL_i^{\pm}$:
\begin{align*}
\Obs_i&=R\HHom_\pi(\OO_{\ZZ_i}, \mathcal{P}_i)\otimes \LL_i,\\
N_{ij}&=R\HHom_\pi(I_{\ZZ_i}\!\otimes\!\mathcal{P}_i^{-1},\OO)\otimes \LL_i
-\!R\HHom_\pi(I_{\ZZ_i}\!\otimes\!\mathcal{P}_i^{-1},I_{\ZZ_j}\!\otimes\!\mathcal{P}_j^{-1})\otimes \LL_i\otimes \LL_j^{-1},\\
T_{X^{[m_i]}}&=T_{X^{[m_i]}},\\
T^{\vir}_{\Hilb_X^{\beta_i}}&=(R\pi_*\mathcal{P}_i)\otimes \LL_i-R\pi_*\OO.
\end{align*}
Recall the standard formula 
\begin{equation}\label{chern class formula}
c_k(E\otimes L)=\sum_{a+b=k} {r-a\choose b} c_a(E) \, c_1(L)^b
\end{equation}
for the Chern classes of a rank $r$ vector bundle $E$ and a line bundle $L$. In particular, this expression is a polynomial of $r,\, c_1(L)$ and the Chern classes of $E$. Moreover, the degree of $r$ is bounded by the degree of $c_1(L)$ in each monomial. Note that the ranks of the $K$-theory classes
\[R\HHom_\pi(\OO_{\ZZ_i}, \mathcal{P}_i),\ \ R\HHom_\pi(I_{\ZZ_i}\!\otimes\!\mathcal{P}_i^{-1},\OO), \ \ R\HHom_\pi(I_{\ZZ_i}\!\otimes\!\mathcal{P}_i^{-1},I_{\ZZ_j}\!\otimes\!\mathcal{P}_j^{-1}),\ \ R\pi_*\mathcal{P}_i
\]
are linear combinations of $m_i,\, \chi(\beta_i),\, \chi(\beta_i, \beta_j)$. By the above Chern class formula, this implies that the Chern classes of (\ref{Kclasses step2'}) are polynomials, depending on $\underline{m}$, of $\chi(\beta_i),\, \chi(\beta_i, \beta_j),\, h_i$ and the Chern classes of
\begin{equation}\label{Kclasses step2}
R\HHom_\pi(I_{\ZZ_i}\!\otimes\!\mathcal{P}_i^{-1},\OO),\ \ R\HHom_\pi(I_{\ZZ_i}\!\otimes\!\mathcal{P}_i^{-1},I_{\ZZ_j}\!\otimes\!\mathcal{P}_j^{-1}), \ \
R\pi_*\mathcal{P}_i,\ \ 
T_{X^{[m_i]}}.
\end{equation}
We omitted $R\HHom_\pi(\OO_{\ZZ_i},\mathcal{P}_i)$ in the list because
\[R\HHom_\pi(\OO_{\ZZ_i},\mathcal{P}_i)=R\pi_*\mathcal{P}_i-R\HHom_\pi(I_{\ZZ_i}\otimes \mathcal{P}_i^{-1},\OO).
\]
Therefore one can rewrite the universal power series $\mathsf{A}_{\underline{m}}$ as 
\[ \mathsf{A}_{\underline{m}}=\sum_{k_1, \dots, k_N\geq 0} h_1^{k_1}\cdots h_N^{k_N} \cdot \mathsf{B}_{\underline{k}, \underline{m}},
\]
where $\mathsf{B}_{\underline{k}, \underline{m}}$ is the universal power series of integers $\chi(\beta_i),\, \chi(\beta_i, \beta_j)$ and the Chern classes of the elements in (\ref{Kclasses step2})
with coefficients in $\Q(y, e^{w_1}, \dots, e^{w_N})$. Note that degree of $\chi(\beta_i),\, \chi(\beta_i, \beta_j)$ in $\mathsf{B}_{\underline{k}, \underline{m}}$ is bounded by $\sum_i k_i$ because of the remark below the Chern class formula \eqref{chern class formula}. It is clear from the definition of the Seiberg-Witten invariants that
\[ (\mathsf{AJ}^{\underline{\beta}})_*
\left( h_1^{k_1}\cdots h_N^{k_N} \cap [\Hilb_X^{\underline{\beta}}]^\vir\right)
=\SW^{k_1}(\beta_1)\times\cdots \times \SW^{k_N}(\beta_N)=:\SW^{\underline{k}}(\underline{\beta})\, .
\]
Therefore (\ref{step2'}) becomes 
\begin{equation}\label{step2}
\sum_{k_1, \dots, k_N\geq 0} 
\int_{[X^{[\underline{m]}}]\times
\SW^{\underline{k}}(\underline{\beta}) }
\mathsf{B}_{\underline{k},\underline{m}}
\end{equation}
by the projection formula.
\subsubsection{Step 3.}
In this step, we explain the generalized version of the recursive argument in \cite{EGL} which replaces the product of Hilbert schemes of points in (\ref{step2}) by the product of copies of $X$. Finally we integrate over the product of copies of $X$ to complete the proof of Theorem \ref{tuniv}. 

Given a smooth projective surface $X$ and $F\in K^0(X)$, the original version of the recursive argument in \cite{EGL} concerns integrals of the form \[\int_{X^{[m]}} \mathsf{P}\left(F^{[m]},T_{X^{[m]}}\right)
\]
where $\mathsf{P}$ is a polynomial in the Chern classes of the arguments. Inductively, such integrals can be transferred over $X^{[m-1]}\times X$, then $X^{[m-2]}\times X^2$, and so on. After the $k$'th recursion, one gets an integral
\[\int_{X^{[m-k]}\times X^k}\mathsf{P}'\left(
F^{(j)}, F^{[m-k]}, T_{X}^{(j)}, T_{X^{[m-k]}}, I_{\ZZ}^{(j)}, \OO_{\Delta_{j_1, j_2}}
\right)
\]
for some other polynomial $\mathsf{P}'$ of the Chern classes which only depends on $\mathsf{P}$ and $m,\, k,\, \rk(F)$. Here $\Delta_{j_1, j_2}$ denotes the pull back of the diagonal. We use the superscript $(j)$ to denote the location from which factor of $X^k$ the corresponding $K$-theory classes are pulled back. This notation will be used for the rest of Step 3. 

The recursive argument needs to be generalized to deal with the integral
\begin{equation}\label{step3}
\int_{
[X^{[\underline{m]}}]\times
\SW^{\underline{k}}(\underline{\beta}) }
\mathsf{B}_{\underline{k},\underline{m}}.
\end{equation}
Kool and Thomas \cite[Section 4]{KT} already considered such a generalization when $N=1$.  In our case, we integrate over a product of $N$ Hilbert schemes of points. The subscript $1\leq i\leq N$ will be used to denote a specific Hilbert scheme factor out of the product. Furthermore, we have $K$-theory elements $F_i\in K^0(X\times B_i)$ for some smooth projective varieties $B_i$. In our case, $B_i$ and $F_i$ are $\Pic_X^{\beta_i}$ and the Poincar\'e line bundle $\mathcal{P}_i\in K^0(X\times \Pic_X^{\beta_i})$, respectively. Henceforth, we assume that $\rk(F_i)=1$ for all $i$. 

We denote the tautological sheaves appearing in (\ref{Kclasses step2}) by
\begin{align*}
F_i^{[\![m_i]\!]}&:=R\HHom_\pi(I_{\ZZ_i}\otimes F_i^\vee, \OO)\in K^0(X_i^{[m_i]}\times B_i),\\
(F_i, F_{i'})^{[\![m_i, m_{i'}]\!]}&:=R\HHom_\pi(I_{\ZZ_i}\otimes F_i^\vee, I_{\ZZ_{i'}}\otimes F_{i'}^\vee)\in K^0(X_i^{[m_i]}\times X_{i'}^{[m_{i'}]}\times B_i\times B_{i'})
\end{align*}
for each $i\neq i'$. (The index $j$ in Step 2 is replaced by $i'$ here.) 
With minor modifications, the recursion of \cite{EGL} can be applied to integrals of the form 
\begin{equation}\label{start}
\int_{[X_1^{[m_1]}]\times \cdots \times [X_N^{[m_N]}]\times [B_1]\times \cdots [B_N]}
\mathsf{P}\left(F_i^{[\![m_i]\!]}, (F_i, F_{i'})^{[\![m_i, m_{i'}]\!]}, R\pi_*F_i, T_{X^{[m_i]}}
\right),
\end{equation}
where $\mathsf{P}$ is a power series in the Chern classes of the arguments. We work with a power series rather than a polynomial because the dimensions of $B_i$ are not specified. Only non-trivial modification needed here is to deal with the mixed terms $(F_i, F_{i'})^{[\![m_i, m_{i'}]\!]}$. We show here part of the necessary computations to apply the recursion. We closely follow the notations of \cite{EGL}. Consider a diagram
\begin{center}\begin{tikzcd}
X_i^{[m_i-1, m_i]}\times X_{i'}^{[m_{i'}]}\times B_i\times B_{i'} \arrow[r, "\Phi"] \arrow[d, "\Psi"] & X_i^{[m_i-1]}\times X_{i'}^{[m_{i'}]}\times X \times B_i\times B_{i'} \\
X_i^{[m_i]}\times X_{i'}^{[m_{i'}]}\times B_i\times B_{i'} .                    &  
\end{tikzcd}\end{center}
Over $X_i^{[m_i-1, m_i]}\times X$, we have the universal exact sequence
$$0\rightarrow I_{\ZZ_i}\rightarrow I_{\WW_i}\rightarrow j_*\mathcal{L}\rightarrow 0$$
where the line bundle $\mathcal{L}$ over $X_i^{[m_i-1, m_i]}$ and the morphism $j=(\id, \rho)$ are defined in \cite{EGL}. Using this exact sequence, we obtain 
\begin{align*}
\Psi^*(F_i, F_{i'})^{[\![m_i, m_{i'}]\!]}
&=\Phi^*(F_i, F_{i'})^{[\![m_i-1, m_{i'}]\!]}-R\HHom_\pi(j_*\LL\otimes F_i^\vee, I_{\ZZ_{i'}}\otimes F_{i'}^\vee)
\end{align*}
where we suppressed various pull backs. By relative Serre-duality, we simplify the second term
\begin{align*}
R\HHom_\pi(j_*\LL\otimes F_i^\vee, I_{\ZZ_{i'}}\otimes F_{i'}^\vee)
&=R\HHom_\pi(I_{\ZZ_{i'}}\otimes F_{i'}^\vee, j_*\LL\otimes F_i^\vee\otimes K_X)^\vee\\
&=R\pi_*j_*\Big(\LL\otimes j^*(I_{\ZZ_{i'}}^\vee\otimes F_i^\vee\otimes F_{i'}\otimes K_X)\Big)^\vee\\
&=\LL^\vee\otimes\, j^*(I_{\ZZ_{i'}}\otimes F_i\otimes F_{i'}^\vee\otimes K_X^\vee)\\
&=\LL^\vee\otimes \Phi^*(I_{\ZZ_{i'}}\otimes F_i\otimes F_{i'}^\vee\otimes K_X^\vee).
\end{align*}
This result is analogous to \cite[Lemma 2.1]{EGL} which is a key to the recursion. We leave the rest of the details to the reader. See also \cite[Section 5]{GNY2} for the similar discussion involving the mixed terms. 

After the $k$'th recursion, the integral \eqref{start} becomes
\[\int_{[X_1^{[n_1]}]\times \cdots \times [X_N^{[n_N]}]\times [X^k]\times [B_1]\times \cdots [B_N]}
\mathsf{P}',
\]
where $\mathsf{P}'$ is a power series of the Chern classes of the following $K$-theory classes:
\begin{align*}
&1) \textnormal{ Tautological sheaves: }F_i^{(j)}, F_i^{[\![n_i]\!]}, (F_i, F_{i'})^{[\![n_i, n_{i'}]\!]},\\
&2) \textnormal{ Tangent bundles: }T_{X}^{(j)}, T_{X_i^{[n_i]}},\\
&3) \textnormal{ Universal ideal sheaves: }I_{\ZZ_i}^{(j)},\\
&4) \textnormal{ Diagonals: }\OO_{\Delta_{j_1, j_2}},\\
&5) \textnormal{ Pullbacks of } R\pi_* F_i \textnormal{ from } B_i.
\end{align*}
The end result after $m=\sum m_i$ recursions is an integral 
\[\int_{[X^m]\times [B_1]\times\cdots\times[B_N]}\mathsf{P}''\left(
F_i^{(j)}, T_{X}^{(j)}, \OO_{\Delta_{j_1,j_2}}, R\pi_*F_i
\right)
\]
for some power series $\mathsf{P}''$ of the Chern classes of the $K$-theory elements which only depends on $\mathsf{P}$, $\underline{m}$ and the ranks of $F_i$. Since we assumed that $\rk(F_i)=1$, we ignore the dependence on $\rk(F_i)$.  By using the Chern classes of diagonals, this can be further integrated over the $m$-fold product of $X$ yielding
\[\int_{[B_1]\times\cdots \times[B_N]}\mathsf{P}'''
\Big(\pi_*\big(c_{I_1}(F_1)\cdots c_{I_N}(F_N)\cdot c_J(T_X) \big)\Big),
\]
where $c_I$ denotes the monomial of the product of Chern classes corresponding to the multiset $I$ and $\pi$ denotes the projection map $X\times B_1\times \cdots \times B_N\rightarrow B_1\times \cdots \times B_N$.

We apply the above argument to the integral (\ref{step3}). Since the recursive procedure only affects the factors which are Hilbert scheme of points, no change is needed when we replace $[\Pic_X^{\beta_1}]\times\cdots\times[\Pic_X^{\beta_N}]$ by a product of higher Seiberg-Witten invariants $\SW^{k_1}(\beta_1)\times\cdots\times\SW^{k_N}(\beta_N)$ regarded as a homology class. Therefore the integral (\ref{step3}) becomes
\[\int_{[\SW^{k_1}(\beta_1)]\times\cdots\times[\SW^{k_N}(\beta_N)]} \mathsf{B}_{\underline{k},\underline{m}}''',
\]
where $\mathsf{B}_{\underline{k}, \underline{m}}'''$ is a power series of $\chi(\beta),\, \chi(\beta_i,\beta_j)$ and
\begin{equation}\label{final terms}
\pi_*\big(c_{1}(\mathcal{P}_1)^{\ell_1}\cdots c_1(\mathcal{P}_N)^{\ell_N}\cdot c_J(T_X) \big)
\end{equation}
with coefficients in $\Q(y, e^{w_1}, \dots, e^{w_N})$.  Note that the first Chern class of the Poincar\'e line bundle has a decomposition
\[c_1(\mathcal{P}_i)=(\beta_i,\id_i,0)\in H^2(X)\oplus \Big(H^1(X)\otimes H^1(X)_i^\vee\Big)\oplus \Lambda^2H^1(X)_i^\vee
\]
after the identification $H^1(\Pic_X^{\beta_i})=H^1(X)^\vee_i$. The subscript $i$ again indicates the position among the $N$ Picard group factors. For degree reasons, the classes in (\ref{final terms}) are pushforwards of linear combinations of 
\[c_1(T_X)^2,\ \  c_2(T_X),\ \  \beta_i^2,\ \  \beta_i.c_1(T_X),\ \  \beta_i.\beta_j,\ \  c_1(T_X).\id_s\id_t,\ \  \beta_i.\id_s\id_t,\ \  \id_s\id_t\id_p\id_q\] 
under $\pi_*$. 
These contributions also cover $\chi(\beta_i),\,  \chi(\beta_i, \beta_j)$ by the Riemann-Roch theorem. Recall the definition of $[\alpha]_{s_1\leq \cdots\leq s_k}$ in Paragraph \ref{subsection universality}. One can check, for example by  picking a basis of $H^1(X)$, that 
\[\pi_*(\alpha. \id_{s_1}.\cdots.\id_{s_k})= \textnormal{constant} \cdot [\alpha]_{s_1\leq \cdots\leq s_k},
\] 
where the constant factor is determined by the position of equalities in $s_1\leq \dots\leq s_k$. Therefore, we conclude that (\ref{step3}) becomes
\[\int_{[\SW^{\underline{k}}(\underline{\beta})]}
\mathsf{Q}_{\underline{k},\underline{m}}\left(K_X^2,\, \chi(\mathcal{O}_X),\, \beta_i^2,\,  \beta_i.K_X,\,  \beta_i.\beta_j,\, 
[K_X]_{s\leq t},\,  [\beta_i]_{s\leq t},\,  [1]_{s\leq t\leq p\leq q}\right),
\]
where $\mathsf{Q}_{\underline{k}, \underline{m}}$ is a universal power series, which only depends on $\underline{k}, \underline{m}$, with coefficients in $\Q(y,e^{w_1}, \dots, e^{w_N})$. Furthermore, one can show that $\mathsf{Q}_{\underline{k}, \underline{m}}$ is a polynomial in the first five degree zero cohomology classes by carefully tracing back the argument. This completes the proof. 

\subsection{The multiplicative structure of the Quot scheme invariants}\label{smult}
\noindent In this subsection, we establish the multiplicative structure of the equivariant Quot scheme invariants for a curve class of Seiberg-Witten length $N$. This will be slightly weaker than Theorem \ref{tmult}. The proof of Theorem \ref{tmult} will be completed in Subsection \ref{s0}. 

Recall from Definition \ref{dSW} that a curve class $\beta$ is of Seiberg-Witten length $N$ if for any effective decomposition 
\[\beta=\sum_{i=1}^N \beta_i \quad \textnormal{such that} \quad \SW(\beta_i)\neq 0 \  \textnormal{ for all } i,
\]
we have $\vd_{\beta_i}=0$ for all $i$. By Theorem \ref{tuniv}, each fixed locus corresponding to a decomposition $\beta=\sum \beta_i$ does not contribute to the Quot scheme invariants if $\SW(\beta_i)=0$ for some $i$. Therefore, for a curve class $\beta$ of Seiberg-Witten length $N$, it suffices to consider effective decompositions with $\vd_{\beta_i}=0$ for all $i$. This implies that no higher Seiberg-Witten invariants are involved in the calculation. As we noted in Remark \ref{rSW}, any curve classes on  surfaces with $p_g>0$ will satisfy this property. 

Let $X$ be a smooth projective surface and $\beta$ be a curve class of Seiberg-Witten length $N$. Consider the generating series of the equivariant Quot scheme invariants
\[Z_{X, N,\,\beta}(q\,|\,\underline{w})=\sum_{n\in \Z} q^n \chi_{-y}^\vir\big(\Quot_X(\C^N\!,\beta,n)\big).\]
We already observed that the $\C^{*}$-fixed loci over $\Quot_X(\C^N, \beta, n)$ for all $n$ are indexed by
\[F_{\underline{m},\underline{\beta}}=\Big(X^{[m_1]}\times\cdots \times  X^{[m_N]}\Big)
\times \Big(\Hilb_X^{\beta_1}\times \cdots \times \Hilb_X^{\beta_N}\Big)
\]
for some effective decomposition $\beta=\sum \beta_i$ and $m_1, \dots, m_N\geq 0$. Note that $F_{\underline{m},\underline{\beta}}$ contributes to the Quot scheme invariants only if $\vd_{\beta_i}=0$ for all $i$ by the assumption. In this case, the exponent $n$ of the formal variable $q$ is given by 
\[n=\sum_i\left(m_i-\frac{\beta_i.(\beta_i+K_X)}{2}\right)
=\sum_i \Big(m_i-\vd_{\beta_i}-\beta_i.K_X\Big)
=\left(\sum_i m_i\right)-\beta.K_X.
\]
Therefore the generating series of the equivariant Quot scheme invariants becomes
\begin{equation}\label{eqnSW}
Z_{X, N,\,\beta}(q\,|\,\underline{w})=q^{-\beta.K_X}\!\!\!\!\!\!\!
\sum_{\substack{\beta=\sum \beta_i\\ s.t. \ \vd_{\beta_i}=0}}
\sum_{m_i\geq 0} q^{\sum m_i}
\int_{[X^{[\underline{m}]}]\times[\Hilb_X^{\underline{\beta}}]^\vir} (*)
\end{equation}
with 
\[(*)=\frac{e}{\XX_{-y}}\left(\sum_i \Obs_i\right)\  
\frac{\XX_{-y}}{e}\left(N^\vir_F\right)\ 
\XX_{-y}\left(\sum_i T_{X^{[m_i]}}+T^\vir_{\Hilb_X^{\beta_i}}\right)
\]
as in the integral (\ref{step1}). 

Since we have $\vd_{\beta_i}=0$ for all $i$, we can write
$$[\Hilb_X^{\beta_i}]^\vir=\sum_{j=1}^s n_{ij}[x_{ij}]\in H_0(\Hilb_X^{\beta_i}, \Z)
$$
where $x_{ij}$'s are arbitrarily chosen from each connected components of $\Hilb_X^{\beta_i}$. By taking degree, we know that $\sum_j n_{ij}=\SW(\beta_i)$. For the purpose of computations below, we may assume
$$[\Hilb_X^{\beta_i}]^\vir=\SW(\beta_i)[\mathsf{pt}_i].$$
Pick representatives for the point classes $[\mathsf{pt}_i]$, that is, pick exact sequences
\[0\rightarrow \OO(-D_i)\rightarrow\OO_X\rightarrow\OO_{D_i}\rightarrow0
\]
such that $c_1(\OO(D_i))=\beta_i$. After restriction to these point representatives, the $K$-theory classes involved in $(*)$ become
\begin{align*}
\Obs_i&=R\HHom_\pi(\OO_{\ZZ_i}, \OO(D_i)),\\
N_F^\vir&=\sum_{i\neq j} R\HHom_\pi(I_{\ZZ_i}(-D_i), \OO-I_{\ZZ_j}(-D_j))[w_j-w_i],\\
T^{\vir}_{\Hilb_X^{\beta_i}}&=0.
\end{align*}
Therefore (\ref{eqnSW}) is equal to 
\[q^{-\beta.K_X}\!\!\!\!\!\!\!
\sum_{\substack{\beta=\sum \beta_i\\ s.t. \ \vd_{\beta_i}=0}}
\SW(\beta_1)\cdots\SW(\beta_N)
\sum_{m_i\geq  0} q^{\sum m_i}\int_{[X^{[\underline{m}]}]} (**),
\]
where $(**)$ is the pull back of $(*)$ under the embedding with respect to the point representatives. Denote the contributions from the Hilbert schemes of points with respect to the decomposition $\beta=\sum \beta_i$ as follows:
\begin{equation} \label{eqnlocal}
\widehat{Z}_{X,N,\, \underline{\beta}}(q\,|\, \underline{w})
:=\sum_{m_i\geq 0} q^{\sum m_i} \int_{X^{[\underline{m}]}}
\frac{e}{\XX_{-y}}\left(\sum_i \Obs_i\right)\  
\frac{\XX_{-y}}{e}\left(N^\vir_F\right)\ 
\XX_{-y}\left(\sum_i T_{X^{[m_i]}}\right),
\end{equation}
where 
\begin{equation*}
\begin{aligned}
\Obs_i&=R\HHom_\pi(\OO_{\ZZ_i}, L_i),\\
N_F^\vir&=\sum_{i\neq j} R\HHom_\pi(I_{\ZZ_i}\otimes L_i^{-1}, \OO-I_{\ZZ_j}\otimes  L_j^{-1})[w_j-w_i]
\end{aligned}
\end{equation*}
for some choice of line bundles $L_i$ with $c_1(L_i)=\beta_i$. Note that $\widehat{Z}_{X,N,\,\underline{\beta}}(q\,|\, \underline{w})$ only depends on the cohomology classes $\beta_i$, not on the choice of line bundles $L_i$, hence the notation.

We now establish the multiplicative property of $\widehat{Z}_{X,N,\, \underline{\beta}}(q\,|\, \underline{w})$. This essentially follows from the argument in \cite{EGL}. Define
\[\mathcal{K}_N:=\{(X,\beta_1, \dots, \beta_N)\,|\, 
X \textnormal{ smooth projective surface}, \ \beta_1, \dots, \beta_N\in H^2(X)\}
\]
and let 
\[ \widehat{Z}_{*,N,*}(q\,|\,\underline{w}): \mathcal{K}_N\rightarrow 1+q\,\Q(y,e^{w_1}, \dots, e^{w_N})[[q]]
\]
be the obvious map. This map is well defined even if there are no line bundle representatives for the $\beta_i$, by using the Grothendieck-Riemann-Roch formula. By the generalized version of the recursive argument explained in the proof of Theorem \ref{tuniv}, this map factors through $\gamma_N:\mathcal{K}_N\rightarrow \Q^{t_N}$
\begin{align*}
(X,\beta_1, \dots, \beta_N)\mapsto
\left(K_X^2,\, \chi(\OO_X),\, \beta_i.K_X,\, \beta_i(\beta_i-K_X),\, \beta_i.\beta_j\right)
\end{align*}
followed by 
\[h:\Q^{t_N}\rightarrow 1+q\,\Q(y,e^{w_1}, \dots, e^{w_N})[[q]].
\]
Here $t_N$ is the number of the relevant Chern numbers given by
\[t_N=1+1+N+N+{N\choose 2}.
\]
To conclude the multiplicativity, consider a disjoint union of surfaces
\[(X\sqcup Y, \alpha_1\sqcup \beta_1, \dots, \alpha_N\sqcup \beta_N).
\] 
Since $e$ and $\mathcal{X}_{-y}$ are multiplicative characteristic classes, it follows that 
\[\widehat{Z}_{X\sqcup Y, N,\, \underline{\alpha}\,\sqcup \underline{\beta}}(q\,|\,\underline{w})=
\widehat{Z}_{X,N,\,\underline{\alpha}}(q\,|\,\underline{w})\cdot 
\widehat{Z}_{Y,N,\,\underline{\beta}}(q\,|\,\underline{w}).
\]
By the argument of [EGL], we conclude that there exist universal series
\[\mathsf{U}_N, \quad \widetilde{\mathsf{U}}_N, \quad \mathsf{V}_{N,i},\quad \widetilde{\mathsf{V}}_{N,i} \quad \mathsf{W}_{N,i,j},
\]
in $1+q\cdot\Q(y)(e^{w_1}, \dots, e^{w_N})[[q]]$ such that 
\[\widehat{Z}_{X,N,\,\underline{\beta}}(q\,|\,\underline{w})=\mathsf{U}_N^{\, K_X^2}\ 
\widetilde{\mathsf{U}}_N^{\,\chi(\OO_X)}
\prod_{i=1}^N\mathsf{V}_{N,i}^{\, \beta_i.K_X}
\prod_{i=1}^N\widetilde{\mathsf{V}}_{N,i}^{\ \beta_i.(\beta_i-K_X)}\!\!\!
\prod_{1\leq i<j\leq N}\!\!\!\!\!\mathsf{W}_{N,i,j}^{\, \beta_i.\beta_j}.
\]

Since we have
\[\vd_{\beta_i}=\frac{\beta_i(\beta_i-K_X)}{2}=0 \ \ \text{for all} \ \ i
\]
in our setting, we obtain
\[\widehat{Z}_{X,N,\,\underline{\beta}}(q\,|\,\underline{w})=\mathsf{U}_N^{K_X^2}\ 
\widetilde{\mathsf{U}}_N^{\chi(\OO_X)}
\prod_{i=1}^N\mathsf{V}_{N,i}^{\, \beta_i.K_X}\!\!\!\!\!
\prod_{1\leq i<j\leq N}\!\!\!\!\!\mathsf{W}_{N,i,j}^{\, \beta_i.\beta_j}.
\]
In the following section, we prove that $\widetilde{\mathsf{U}}_N=1$. This will complete the proof of Theorem \ref{tmult}.

\section{Applications of the multiplicative structure} \label{sapp}
\noindent In this section, we use Theorem \ref{tmult} to find some of the universal series in the previous section. We do so by picking convenient geometries. Our choices will involve $K3$ surfaces and surfaces with smooth canonical curves. The formulas we obtain this way will imply Theorems \ref{t0}, \ref{tell}, \ref{tg}, and \ref{tbl}.

\subsection{Punctual quotients}\label{s0}
\noindent In this subsection, we set $\beta=0$, thus studying the invariants of punctual Quot schemes. Note that the zero curve class only admits a trivial effective decomposition by $\beta_i=0$ for all $i$. On the other hand, we have $\SW(0)=1$ because $\Hilb_X^{\beta=0}$ is a reduced point of virtual dimension 0. Since $\beta_i=0$ for all $i$, the weaker form of Theorem \ref{tmult} in Subsection \ref{smult} implies
\[Z_{X,N,\,0}(q\,|\,\underline{w})=\widehat{Z}_{X,N,\,(0,\cdots, 0)}(q\,|\, \underline{w})
=\mathsf{U}_N^{\, K_X^2}\ \widetilde{\mathsf{U}}_N^{\, \chi(\OO_X)}.
\]
In what follows, we use special geometries to determine $\mathsf{U}_N$ and $\widetilde{\mathsf{U}}_N$. 

 We prove that $\widetilde{\mathsf{U}}_N=1$. To this end, let $X$ be a $K3$ surface. Recall from (\ref{eqnlocal}) that
\[\widehat{Z}_{X,N,\,(0,\dots,0)}(q\,|\,\underline{w})=\!\!
\sum_{m_i\geq 0} q^{\sum m_i} \!
\int_{X^{[\underline{m}]}}
\frac{e}{\XX_{-y}}\left(\sum_i \Obs_i\!\right) 
\frac{\XX_{-y}}{e}\left(N^\vir_F\right)
\XX_{-y}\left(\sum_i T_{X^{[m_i]}}\right)\!,
\]
where
\begin{align*}
\Obs_i&=R\HHom_\pi(\OO_{\ZZ_i}, L_i),\\
N_F^\vir&=\sum_{i\neq j} R\HHom_\pi(I_{\ZZ_i}\otimes L_i^{-1}, \OO-I_{\ZZ_j}\otimes  L_j^{-1})[w_j-w_i]
\end{align*}
for some line bundles $L_i$ with $c_1(L_i)=0$. We take $L_i=\OO_X$ for all $i$. Then we have
\[\Obs_i=R\HHom_\pi(\OO_{\ZZ_i}, \OO)=\big(R\pi_*\OO_{\ZZ_i}(K_X)\big)^\vee=(K_X^{[m_i]})^\vee
\]
by Serre duality. Since $X$ is a $K3$ surface, this becomes $(\OO_X^{[m_i]})^\vee$. Its Euler class vanishes if $m_i>0$ due to a trivial summand. This proves that 
\[\widehat{Z}_{X, N,\,(0,\dots,0)}(q\,|\,\underline{w})=1.
\]
On the other hand, we have 
\[\big(\mathsf{U}_N\big)^0 \big(\widetilde{\mathsf{U}}_N\big)^2=\widehat{Z}_{X, N,\,(0,\dots,0)}(q\,|\,\underline{w})=1
\]
since $K_X^2=0$ and $\chi(\OO_X)=2$. This implies $\widetilde{\mathsf{U}}_N=1$ and completes the proof of Theorem \ref{tmult}.

 Next, we study $\mathsf{U}_N$ by using a surface $X$ with a smooth canonical curve $C\in |K_X|$. If $C$ is cut out by a section $s\in H^0(X,K_X)$, then the tautological section $s^{[m_i]}\in H^0(X^{[m_i]}, K_X^{[m_i]})$ vanishes on $C^{[m_i]}\subset X^{[m_i]}$. 
Therefore
\[e(\Obs_i)=e\big((K_X^{[m_i]})^\vee\big)=(-1)^{m_i}\, e\big(K_X^{[m_i]}\big)=(-1)^{m_i}[C^{[m_i]}].
\]
Denote the embedding of a product of Hilbert scheme of points on $C$ by
\[i_{\underline{m}}:C^{[m_1]}\times\cdots\times C^{[m_N]}\hookrightarrow
X^{[m_1]}\times\cdots\times X^{[m_N]}.
\]
Then $\widehat{Z}_{X, N,\,(0,\dots,0)}(q\,|\,\underline{w})$ becomes
\begin{equation}\label{eqn0}
\sum_{m_i\geq0}(-q)^{\sum m_i} \int_{C^{[\underline{m}]}}
\mathcal{X}_{-y}\left(i_{\underline{m}}^*\Big(\sum_i T_{X^{[m_i]}}-\Obs_i\Big)\right)
\frac{\mathcal{X}_{-y}}{e}\left(i_{\underline{m}}^* N^\vir_F\right).
\end{equation}

We compute the pull back of the $K$-theory classes in (\ref{eqn0}) in terms of the universal data of Hilbert scheme of points on $C$ and the normal bundle 
\[\Theta:=N_{C/X}=\OO_C(C).\]
Since $C$ is a canonical curve of $X$, $\Theta$ is a theta characteristic of $C$, that is, $\Theta^{\otimes 2}=\omega_C$. The usual normal bundle exact sequence associated to the tautological section $s^{[\underline{m}]}\in H^0(X^{[\underline{m}]},K_X^{[\underline{m}]})$ gives
\begin{align*}
i_{\underline{m}}^* T_{X^{[\underline{m}]}}&=T_{C^{[\underline{m}]}}+i_{\underline{m}}^* K_X^{[\underline{m}]}
=T_{C^{[\underline{m}]}}+\Theta^{[\underline{m}]}
\end{align*}
in $K$-theory and so
\begin{equation}\label{pullbackT}
i_{\underline{m}}^*\Big(\sum_i T_{X^{[m_i]}}-\Obs_i\Big)=T_{C^{[\underline{m}]}}+\Theta^{[\underline{m}]}-(\Theta^{[\underline{m}]})^\vee.
\end{equation}
This recovers a formula from \cite{OP}. Now we compute the pull back of the virtual normal bundle. For this calculation, we change the notation $R\HHom_\pi$ to $\Ext_X^\bullet$ to indicate it is the relative pushforward along the projection with fiber $X$. Similarly, the relative pushforward along the projection with fiber $C$ will be denoted $\Ext^\bullet_C$. With these notations, 
\begin{align}\label{pullbackN}
i_{\underline{m}}^*N_F^\vir&=\sum_{i\neq j} \Ext_X^\bullet(I_{\ZZ_i}, \OO_{\ZZ_j})[w_j-w_i]\nonumber\\
&=\sum_{i\neq j} \Big(\Ext_X^\bullet(\OO, \OO_{\ZZ_j})-\Ext_X^\bullet(\OO_{\ZZ_i}, \OO_{\ZZ_j})\Big)[w_j-w_i]\nonumber\\
&=\sum_{i\neq j} \Big( \Ext_C^\bullet(\OO, \OO_{\ZZ_j})-\Ext_C^\bullet(\OO_{\ZZ_i}, \OO_{\ZZ_j})+\Ext_C^\bullet(\OO_{\ZZ_i}, \OO_{\ZZ_j}\otimes \Theta)\Big)[w_j-w_i]\nonumber\\
&=\sum_{i\neq j} \Big( \Ext_C^\bullet(I_{\ZZ_i/C}, \OO_{\ZZ_j})+\Ext_C^\bullet(\OO_{\ZZ_i}, \OO_{\ZZ_j}\otimes \Theta)\Big)[w_j-w_i],
\end{align}
where $\ZZ_i$ is the universal zero dimensional subscheme, scheme theoretically supported on $C$. The third equality above follows from \cite[Lemma 3.42]{T1}. Note that both of these pull back of $K$-theory classes are intrinsic to the curve $C$ up to the choice of a theta characteristic. This  gives another way to explain that $\widetilde{\mathsf{U}}_N=1$ since the Chern number of $C$ is determined by $g(C)-1=K_X^2$ which is independent of $\chi(\OO_X)$. 

Now we further restrict the surface geometry to compute the integral (\ref{eqn0}). Assume that $X$ has a smooth canonical curve $C$ of genus 0. For example, the blow up of a $K3$ surface at a point will satisfy this. When $C=\PP^1$, we can identify
\[C^{[m_1]}\times \cdots \times C^{[m_N]}=\PP^{m_1}\times\cdots\times\PP^{m_N}=:\PP^{\underline{{m}}}.
\]
With this identification, the universal exact sequence is given by
\[0\rightarrow \sum_i\OO_{\PP^{m_i}}(-1)\boxtimes \OO_{\PP^1}(-m_i)
\rightarrow \sum_i \OO_{\PP^{\underline{m}}\times\PP^1}\rightarrow \sum_i\OO_{\ZZ_i}\rightarrow 0.
\]
This, in particular, implies the formula for the tautological bundle over $\PP^m$:
\begin{equation}\label{taut. bundle}\OO_{\PP^1}(d)^{[m]}=\OO_{\PP^m}(-1)^{\oplus (m-d-1)}+\OO_{\PP^m}^{\oplus (d+1)}.
\end{equation}
Using this formula, the $K$-theory classes of \eqref{pullbackT} and \eqref{pullbackN} are computed in the proof of \cite[Theorem 23]{OP}:
\begin{align*}
i_{\underline{m}}^*\Big(\sum_i T_{X^{[m_i]}}-\Obs_i\Big)&=\sum_i\Big(\OO(-h_i)^{\oplus m_i}+\OO(h_i)-\OO\Big),\\
i_{\underline{m}}^*N_F^\vir&=\sum_{i\neq j} \Big(\OO(-h_j)^{\oplus m_j}+\OO(h_i)-\OO(h_i-h_j)\Big)[w_j-w_i],
\end{align*}
where $h_i$ is the hyperplane section of $\PP^{m_i}$. We explain how \cite{OP} obtains these formulas by showing the part of the computation. For example, we compute the part of the terms in \eqref{pullbackN}
\begin{align*}
\Ext_{\PP^1}^\bullet(I_{\ZZ_i/\PP^1}, \OO_{\ZZ_j})
&=\Ext_{\PP^1}^\bullet\Big(\OO_{\PP^{\underline{m}}}(-h_i)\boxtimes \OO_{\PP^1}(-m_i), 
\OO_{\PP^{\underline{m}}\times\PP^1}-\OO_{\PP^{\underline{m}}}(-h_j)\boxtimes \OO_{\PP^1}(-m_j)\Big)\\
&=\OO_{\PP^{\underline{m}}}(h_i)\otimes H^\bullet(\OO_{\PP^1}(m_i))-\OO_{\PP^{\underline{m}}}(h_i-h_j)\otimes H^\bullet(\OO_{\PP^1})\\
&=\OO(h_i)^{\oplus (m_i+1)}-\OO(h_i-h_j)
\end{align*}
by the identification of the universal sequence and Riemann-Roch calculations. 

Recall from Subsection \ref{sssss} that for any line bundle $L$ with $c_1(L)=\alpha$,
\[\XX_{-y}(L)=\frac{\alpha(1-ye^{-\alpha})}{1-e^{-\alpha}}.
\]
For simplicity, we denote
\[\XX_{-y}(\alpha):=\frac{\alpha(1-ye^{-\alpha})}{1-e^{-\alpha}}.
\]
One can check that $\XX_{-y}(0)=1-y$. Substituting the computations of the $K$-theory classes into (\ref{eqn0}), $\mathsf{U}_N^{\, -1}=\mathsf{U}_N^{\, K_X^2}$ becomes
\begin{multline*}
\sum_{m_i\geq 0} (-q)^{\sum m_i}
\int_{\PP^{m_1}\times \cdots \times \PP^{m_N}} 
\prod_i \Big(\mathcal{X}_{-y}(-h_i)^{m_i} \cdot \mathcal{X}_{-y}(h_i) \cdot \frac{1}{1-y}\Big)\\
\times \prod_{i\neq j} \Big(
\frac{\mathcal{X}_{-y}(-h_i+w_i-w_j)^{m_i}}{(-h_i+w_i-w_j)^{m_i}} \cdot
\frac{\mathcal{X}_{-y}(h_i-w_i+w_j)}{(h_i-w_i+w_j)} \cdot \frac{(h_i-h_j+w_j-w_i)}{\mathcal{X}_{-y}(h_i-h_j+w_j-w_i)} \Big).
\end{multline*}
Integrating over the products of projective spaces,  the above expression can be rewritten as
\[\frac{1}{(1-y)^N} \sum_{m_i\geq 0} x_1^{m_1} \cdots x_N^{m_N} [h_1^{m_1} \cdots h_N^{m_N}]\ \Phi_1^{m_1} \cdots \Phi_N^{m_N}\cdot \Psi,
\]
where $-q=x_1=\cdots =x_N$ and bracket extracts the coefficient of the specified monomial. Here 
\begin{align*}
\Phi_i&=\mathcal{X}_{-y}(-h_i)\cdot \prod_{j\neq i} \frac{\mathcal{X}_{-y}(-h_i+w_i-w_j)}{(-h_i+w_i-w_j)},\\
\Psi&=\prod_i \mathcal{X}_{-y}(h_i) \cdot \prod_{i\neq j} \Big(\frac{\mathcal{X}_{-y}(h_i-w_i+w_j)}{(h_i-w_i+w_j)} \cdot \frac{(h_i-h_j+w_j-w_i)}{\mathcal{X}_{-y}(h_i-h_j+w_j-w_i)}\Big).
\end{align*}

Following the key idea of \cite{OP}, we apply the multivariate Lagrange-B\"{u}rmann formula of \cite{WW}. This yields
\[\frac{1}{(1-y)^N} \sum_{m_i\geq 0} x_1^{m_1} \cdots x_N^{m_N} [h_1^{m_1} \cdots h_N^{m_N}]\ \Phi_1^{m_1} \cdots \Phi_N^{m_N}\cdot \Psi
=\frac{1}{(1-y)^N} \cdot \frac{\Psi}{K}(h_1, \dots, h_N)\]
for the change of variables
\[ q=-x_i=\frac{-h_i}{\Phi_i}=\prod_{k=1}^N \frac{-h_i+w_i-w_k}{\XX_{-y}(-h_i+w_i-w_k)}\quad s.t.\quad  h_i(q=0)=0\]
and for 
\[K=\prod_{i} \Big(1-h_i \cdot \frac{\frac{d}{dh_i}\Phi_i}{ \Phi_i}\Big).\]
Note that the change of variable formulas for each $h_1, \dots, h_N$ are given by different equations due to the presence of $w_i$ in each formula. We can however work with a single equation after the additional change of variables
\[h_i:=H_i+w_i-w_1
\]
also used in \cite{OP}. Now, $H_1, \dots, H_N$ are the solutions of a single equation
\begin{equation}\label{eqncov}
q=\prod_{k=1}^N \frac{-H+w_1-w_k}{\XX_{-y}(-H+w_1-w_k)}=
\prod_{k=1}^N\frac{1-e^{H-w_1+w_k}}{1-ye^{H-w_1+w_k}}
\quad s.t. \quad H_i(q=0)=w_1-w_i.
\end{equation}
Straightforward calculation shows that, after the change of variables above, 
\[\frac{1}{(1-y)^N} \cdot \frac{\Psi}{K}(H_1,H_2+w_2-w_1, \dots, H_N+w_N-w_1)
\]
becomes
\begin{multline*}
\frac{(-1)^N}{(1-y)^{2N}}
\prod\limits_{i, j}\frac{1-ye^{-H_i+w_1-w_j}}{1-e^{-H_i+w_1-w_j}} \cdot
\prod\limits_{i\neq j} \frac{1-e^{-H_i+H_j}}{1-ye^{-H_i+H_j}}\\
\times\prod\limits_i\left(\sum\limits_j \frac{e^{H_i-w_1+w_j}}{(1-ye^{H_i-w_1+w_j})(1-e^{H_i-w_1+w_j})}\right)^{-1}.
\end{multline*}

This is a rational function in the variables $e^{H_1}, \dots, e^{H_N}$, $y$, $e^{w_1}, \dots, e^{w_N}$ which is symmetric in $H_1, \dots, H_N$. Since the elementary symmetric functions of $e^{H_1}, \dots, e^{H_N}$ are rational in the variables $q, y, e^{w_1}, \dots, e^{w_N}$ by (\ref{eqncov}), we have shown that 
\[\mathsf{U}_N\in \Q(y, e^{w_1}, \dots, e^{w_N})(q).
\]
Furthermore, it is clear that $\mathsf{U}_N$ does not have poles at $e^{w_1}=\cdots =e^{w_N}=1$ or equivalently at $w_1=\cdots=w_N=0$. 

We specialize at $w_1=\cdots=w_N=0$ to obtain non equivariant answer for $Z_{X,N}(q)$. From the above formula, $\mathsf{U}_N^{-1}\Big\vert_{w_1=\cdots=w_N=0}$ is
\[\frac{(-1)^N}{(1-y)^{2N}\, N^N}
\prod\limits_{i}\left(\frac{1-ye^{-H_i}}{1-e^{-H_i}}\right)^N
\prod\limits_{i\neq j} \frac{1-e^{-H_i+H_j}}{1-ye^{-H_i+H_j}}\
\prod\limits_i\left(\frac{e^{H_i}}{(1-ye^{H_i})(1-e^{H_i})}\right)^{-1}\!\!,
\]
where $e^{H_1}, \dots, e^{H_N}$ are the distinct roots of the polynomial equation 
\[q=\left(\frac{1-e^H}{1-ye^H}\right)^N\quad \iff\quad (1-e^H)^N-q(1-ye^H)^N=0\ ,
\]
considering $e^H$ as a variable. To simplify further, we use another change of variables
\[t_i=\frac{1-e^{H_i}}{1-ye^{H_i}}, \ \ \ i=1, \dots, N.\]
One can check that $\overline{\mathsf{U}}_N:=\mathsf{U}_N\Big\vert_{w_1=\cdots=w_N=0}$ becomes
\[(-1)^N \, N^N
\prod_i\left(\frac{-t_i}{1-(1+y)t_i}\right)^N 
\prod_i\frac{(1-t_i)(1-yt_i)}{t_i}\,
\prod_{i\neq j} \frac{1-(1+y)t_i+yt_it_j}{t_j-t_i},
\]
where $t_1, \dots, t_N$ are the distinct roots of the equation
\[q=t^N.\]
Let $f(t)$ be a polynomial whose roots are $t_1, \dots, t_N$, that is,
\[f(t)=\prod_i (t-t_i)=t^N-q.\]
We can easily compute the symmetric expressions appearing in $\overline{\mathsf{U}}_N$ in terms of $f(t)$. 
\begin{align*}
(\textnormal{\romannumeral 1})\ &\prod_i\frac{-t_i}{1-(1+y)t_i}=\frac{f(0)}{(1+y)^N\, f((1+y)^{-1})}=\frac{-q}{1-(1+y)^Nq}\\
(\textnormal{\romannumeral 2})\ &\prod_i\frac{(1-t_i)(1-yt_i)}{t_i}=\frac{f(1)\, \big(y^N\, f(y^{-1})\big)}{(-1)^N\, f(0)}=\frac{(1-q)(1-y^Nq)}{(-1)^N(-q)}\\
(\textnormal{\romannumeral 3})\ &\prod_{i\neq j}(t_i-t_j)=\prod_{i} f'(t_i)=\prod_iNt_i^{N-1}=\frac{(Nq)^N}{(-1)^N\, f(0)}=N^N\, (-q)^{N-1}
\end{align*}
Substituting these, we obtain
\[\overline{\mathsf{U}}_N=\frac{(1-q)(1-y^Nq)}{\big(1-(1+y)^Nq\big)^N} \cdot \prod_{i\neq j}\big(1-(1+y)t_i+yt_it_j\big)\]
as required. At $y=1$, this recovers the formula proven in \cite{OP}. The simple functional equation (\ref{functional}) follows easily from here. 

\subsection{Relatively minimal elliptic surfaces}\label{sell}
\noindent In this subsection, we study the Quot scheme invariants of relatively minimal elliptic surfaces. We will derive Theorem \ref{tell} from Theorem \ref{tmult}.

Let $p:X\rightarrow C$ be a relatively minimal elliptic surface over a smooth projective curve. Denote the generic fiber by $F$ and multiple fibers by $F_1, \dots, F_r$ with multiplicity $m_1, \dots, m_r$. Since $p$ is relatively minimal, $K_X$ is a rational multiple of $F$ by the canonical bundle formula
\[K_X=p^*(\omega_C+\mathbb{D})+\sum_j (m_j-1)F_j,
\]
where $\mathbb{D}$ is a divisor of degree $\chi(\OO_X)$. 

Let $\beta$ be a curve class supported on the fibers, that is, $\beta.F=0$. Let $\beta=\sum_{i=1}^N \beta_i$ be an effective decomposition. Since the fiber class $F$ is nef, we have $\beta_i.F\geq 0$ for all $i$. This concludes $\beta_i.F=0$ hence $\beta_i.K_X=0$ for all $i$. By Zariski's Lemma \cite{BHPV},
\[\vd_{\beta_i}=\frac{\beta_i(\beta_i-K_X)}{2}=\frac{\beta_i^2}{2}\leq 0
\]
with equality if and only if $\beta_i$ is a rational multiple of $F$. Therefore $\beta$ is a Seiberg-Witten class of arbitrary length $N$. By Theorem \ref{tmult}, we have
\[Z_{X,N,\,\beta}(q\,|\,\underline{w})=q^{-\beta.K_X} \!\!\!\!\!\!\!\!\!\!\!\!\!\!\!\!\!\!
\sum_{\substack{\beta=\sum \beta_i, \\ \beta_i \textnormal{ rational multiple of F}}}\!\!\!\!\!\!\!\!\!\!\!\!\!\!\!
\SW(\beta_1)\cdots \SW(\beta_N) \cdot \widehat{Z}_{X, N,\,\underline{\beta}}(q\,|\,\underline{w}),
\]
where
\[\widehat{Z}_{X, N,\,\underline{\beta}}(q\,|\,\underline{w})=\mathsf{U}_N^{K_X^2}\ 
\prod_{i=1}^N\mathsf{V}_{N,i}^{\, \beta_i.K_X}\!\!\!\!\!
\prod_{1\leq i<j\leq N}\!\!\!\!\!\mathsf{W}_{N,i,j}^{\, \beta_i.\beta_j}.
\]
On the other hand, since $K_X$ and all classes $\beta_i$ are rational multiples of the fiber, we have 
\[\beta.K_X\,=\,K_X^2\,=\,\beta_i.K_X\,=\,\beta_i.\beta_j\,=\,0.
\]
Therefore 
$\widehat{Z}_{X, N,\,\underline{\beta}}(q\,|\,\underline{w})=1$
and so
\[Z_{X,N,\,\beta}(q)=\!\!\!\!\!\!\!\!\!\!\!\!\!\!\!\!\!\!
\sum_{\substack{\beta=\sum \beta_i, \\ \beta_i \textnormal{ rational multiple of F}}}\!\!\!\!\!\!\!\!\!\!\!\!\!\!\!
\SW(\beta_1)\cdots \SW(\beta_N).\]
In particular, if $\beta$ itself is not a rational multiple of $F$, then invariants all vanish. This completes the proof. 
\subsection{Minimal surfaces of general type with $p_g>0$}\label{sgl}
\noindent In this subsection, we study the Quot scheme invariants of minimal surfaces of general type with $p_g>0$. All such surfaces $X$ are of Seiberg-Witten simple type by  \cite[Proposition 4.20]{DKO} with only two basic classes, namely $0$ and $K_X$. Furthermore, it is shown in \cite{CK} that $\SW(K_X)=(-1)^{\chi(\OO_X)}$. We compute the Quot scheme invariants taking advantage of the complete understanding of the Seiberg-Witten theory of $X$. 

Let $X$ be a minimal surface of general type with $p_g>0$ and $\beta$ be a curve class. Since the only Seiberg-Witten basic classes of $X$ are $0$ and $K_X$, each effective decomposition $\beta=\sum \beta_i$ contributes to the Quot scheme invariants only if $\beta_i=0$ or $\beta_i=K_X$ for all $i$. In particular, the Quot scheme invariants vanish unless $\beta=\ell K_X$ for some $0\leq \ell \leq N$. Assume that $\beta=\ell K_X$ for some $0\leq\ell\leq N$. Note that effective decompositions $\beta=\sum \beta_i$ by basic classes $0$ and $K_X$ are enumerated by partitions $I\sqcup J$ of the set $[N]=:\{1, \dots, N\}$ with $|I|=\ell$. Denote the vector of the effective decomposition corresponding to the partition $I\sqcup J$ as
\[\underline{K}_{I\sqcup J}:=(\beta_1, \dots, \beta_N) ,\quad \textnormal{where} \quad \beta_i=
    \begin{cases}
      K_X, & \text{if}\ i\in I\\
      \ \ 0\ \ , & \text{if}\ i\in J.
    \end{cases}
\]
By Theorem \ref{tmult},
\begin{align*} Z_{X,N,\,\ell K_X}(q\,|\,\underline{w})&=q^{-\ell K_X^2}\cdot \SW(K_X)^{\ell}
\sum_{\substack{I\sqcup J=[N], \\ s.t. \ |I|=\ell}} 
\widehat{Z}_{X,N,\, \underline{K}_{I\sqcup J}}(q\,|\, \underline{w}),
\end{align*}
where
\begin{equation}\label{A}
\widehat{Z}_{X,N,\,\underline{K}_{I\sqcup J}}(q\,|\, \underline{w})\ =\ 
\mathsf{U}_N^{K_X^2}\ 
\prod_{i=1}^N\mathsf{V}_{N,i}^{\, \beta_i.K_X}\!\!\!\!\!\!
\prod_{1\leq i_1<i_2\leq N}\!\!\!\!\!\mathsf{W}_{N,i_1,i_2}^{\, \beta_{i_1}.\beta_{i_2}}\nonumber \ =\ 
\Bigg(\mathsf{U}_N\ \prod_{i\in I} \mathsf{V}_{N,i}
\prod_{\substack{i_1<i_2\\ i_1, i_2\in I}}\!\!\mathsf{W}_{N,i_1,i_2}\Bigg)^{K_X^2}.\nonumber
\end{equation}

Now we study the universal series 
\[\Bigg(\mathsf{U}_N\ \prod_{i\in I} \mathsf{V}_{N,i}
\prod_{\substack{i_1<i_2\\ i_1, i_2\in I}}\!\!\mathsf{W}_{N,i_1,i_2}\Bigg)\]
by using special geometries. Assume that $X$ is a minimal general type surface with a smooth canonical divisor $C\in|K_X|$. In this section, we make the convention that indices in $I$, $J$, and $[N]$ are denoted by $i, j, $ and $k$, respectively. Recall from (\ref{eqnlocal}) that
\[\widehat{Z}_{X,N,\,\underline{K}_{I\sqcup J}}(q\,|\, \underline{w})
=\!\!\sum_{m_k\geq 0} q^{\sum m_k}\!\! \int_{X^{[\underline{m}]}}
\frac{e}{\XX_{-y}}\!\left(\sum_{k=1}^N \Obs_k\!\right)
\frac{\XX_{-y}}{e}\!\left(N^\vir_F\right)
\XX_{-y}\!\left(\sum_{k=1}^N T_{X^{[m_k]}}\!\right)\!,
\]
where 
\begin{align*}
\Obs_i&=R\HHom_\pi(\OO_{\ZZ_i}, K_X)=(\OO_X^{[m_i]})^\vee,\\
\Obs_j&=R\HHom_\pi(\OO_{\ZZ_j}, \OO_X)=(K_X^{[m_j]})^\vee
\end{align*}
by taking $L_i=K_X$ and $L_j=\OO_X$. 
If $m_i>0$ for some $i\in I$, then the integral vanishes because $(\OO_X^{[m_i]})^\vee$ has a trivial summand. Therefore, it suffices to consider the cases with $m_i=0$ for all $i\in I$. Now $\widehat{Z}_{X,N,\, \underline{K}_{I\sqcup J}}(q\,|\, \underline{w})$ becomes
\[\sum_{m_j\geq 0} q^{\sum m_j}\!\! \int_{\prod\limits_{j\in J} X^{[m_j]}}
\frac{e}{\XX_{-y}}\left(\sum_j \Obs_j\right)
\frac{\XX_{-y}}{e}\left(N^\vir_F\right)
\XX_{-y}\left(\sum_j T_{X^{[m_j]}}\right)
\]
with 
\begin{align*}
N^\vir_F&=\sum_{i_1\neq i_2} R\HHom_\pi(\OO_X(-C), \OO_C)[w_{i_2}-w_{i_1}]\\
&+\sum_{j_1\neq j_2} R\HHom_\pi(I_{\ZZ_{j_1}}, \OO_{\ZZ_{j_2}})[w_{j_2}-w_{j_1}]\\
&+\,\sum_{i, j} \Big(R\HHom_\pi(\OO_X(-C), \OO_{\ZZ_{j}})[w_j-w_i]+
R\HHom_\pi(I_{\ZZ_j}, \OO_C)[w_i-w_j]\Big).
\end{align*}
Following the proof of Theorem \ref{t0}, we have
\[e(\Obs_j)=(-1)^{m_j}[C^{[m_j]}]
\]
and denote the embedding by
\[i_{\underline{m}}:\prod_j C^{[m_j]}\hookrightarrow \prod_j X^{[m_j]}.
\]
Then $\widehat{Z}_{X,N,\, \underline{K}_{I\sqcup J}}(q\,|\, \underline{w})$ becomes
\begin{equation}\label{eqncurve}
\sum_{m_j\geq0}(-q)^{\sum m_j} \int_{\prod\limits_{j\in J} C^{[m_j]}}
\mathcal{X}_{-y}\left(i_{\underline{m}}^*\Big(\sum_j T_{X^{[m_j]}}-\Obs_j\Big)\right)
\frac{\mathcal{X}_{-y}}{e}\left(i_{\underline{m}}^* N^\vir_F\right).
\end{equation}
Note that the computation in (\ref{pullbackT}) directly gives
\[i_{\underline{m}}^*\Big(\sum_j T_{X^{[m_j]}}-\Obs_j\Big)=\sum_jT_{C^{[m_j]}}+\Theta^{[m_j]}-(\Theta^{[m_j]})^\vee.
\]
Furthermore, the computation in (\ref{pullbackN}) also applies to the components of $i_{\underline{m}}^*N^\vir_F$ corresponding to the summation over $j_1\neq j_2$. We calculate the remaining terms by using the same method and notation as before
\begin{align*}
\Ext_X^\bullet(\OO_X(-C),\OO_C)&=H^\bullet(\OO_C(C))=H^\bullet(\Theta)=0,\\
\Ext_X^\bullet(\OO_X(-C), \OO_{\ZZ_j})&=\Theta^{[m_j]},\\
\Ext_X^\bullet(I_{\ZZ_j}, \OO_C)&=\Ext_C^\bullet(I_{\ZZ_j/C}, \OO_C)-(\Theta^{[m_j]})^\vee.
\end{align*}
Therefore, 
\begin{align*}
i_{\underline{m}}^*N^\vir_F&=\sum_{j_1\neq j_2} \Big( \Ext_C^\bullet(I_{\ZZ_{j_1}/C}, \OO_{\ZZ_{j_2}})+\Ext_C^\bullet(\OO_{\ZZ_{j_1}}, \OO_{\ZZ_{j_2}}\otimes \Theta)\Big)[w_{j_2}-w_{j_1}]\\
&\ +\sum_{i, j} \bigg(\Theta^{[m_j]}[w_j-w_i]+\Big(\Ext_C^\bullet(I_{\ZZ_j/C}, \OO_C)-(\Theta^{[m_j]})^\vee\Big)[w_i-w_j]\bigg).
\end{align*}
These expressions are intrinsic to the curve $C$ up to the choice of theta characteristic. Therefore we can think of (\ref{eqncurve}) as a generating series of tautological integrals over symmetric products of $C$. By modifying the recursive argument of \cite{EGL} to the symmetric products of a curve, the generating series only depends on $g(C)-1=K_X^2$. We already know from Theorem \ref{tmult} that this dependence is multiplicative. Therefore, we may assume $C=\PP^1$ and repeat the computation in Subsection \ref{s0}. Even though a minimal surface of general type will never have a rational curve as a smooth canonical divisor, we are still able to reduce the computation of (\ref{eqncurve}) to $C=\PP^1$ since the integral is now intrinsic to the curve $C$. 

From here, we closely follow the lines in the proof of Theorem \ref{t0}. Assume that $C=\PP^1$. By the computation in the proof of \cite[Theorem 23]{OP}, we have the following equalities 
\begin{align*}
i_{\underline{m}}^*\Big(\sum_j T_{X^{[m_j]}}-\Obs_j\Big)&=\sum_j\Big(\OO(-h_j)^{\oplus m_j}+\OO(h_j)-\OO\Big),\\
i_{\underline{m}}^*N_F^\vir&=\sum_{j_1\neq j_2} \Big(\OO(-h_{j_2})^{\oplus m_{j_2}}+\OO(h_{j_1})-\OO(h_{j_1}-h_{j_2})\Big)[w_{j_2}-w_{j_1}]\\
&\ +\sum_{i, j}\Big(\OO(-h_j)^{\oplus m_j}[w_j-w_i]+\OO(h_j)[w_i-w_j]\Big)
\end{align*}
in the $K$-theory of $\prod_j\PP^{m_j}$, where $h_j$ is the hyperplane section of $\PP^{m_j}$. Substituting this into (\ref{eqncurve}), we obtain that
\[\Bigg(\mathsf{U}_N\ \prod_{i\in I} \mathsf{V}_{N,i}
\prod_{\substack{i_1<i_2\\ i_1, i_2\in I}}\!\!\mathsf{W}_{N,i_1,i_2}\Bigg)^{-1}\]
is equal to 
{
\begin{align*}
&\sum_{m_j\geq 0} (-q)^{\sum m_j} \int_{\prod\limits_j \PP^{m_j}}
\prod\limits_j \left(
\mathcal{X}_{-y}(-h_j)^{m_j} \cdot \mathcal{X}_{-y}(h_j)\cdot \frac{1}{1-y}
\right)\\
&\times\!\! \prod\limits_{j_1\neq j_2}\!\!\!\left(
\frac{\!\mathcal{X}_{-y}(-h_{j_1\!}+w_{j_1\!}-w_{j_2})^{m_{j_1}}}{(-h_{j_1}+w_{j_1}-w_{j_2})^{m_{j_1}}}\!\cdot\!
\frac{\!\mathcal{X}_{-y}(h_{j_1\!}-w_{j_1\!}+w_{j_2})}{(h_{j_1}-w_{j_1}+w_{j_2})}\!\cdot\!
\frac{(h_{j_1}-h_{j_2}-w_{j_1}+w_{j_2})}{\!\mathcal{X}_{-y}(h_{j_1\!}-h_{j_2\!}-w_{j_1\!}+w_{j_2})}
\!\right)\\
&\times \prod\limits_{i, j} \left(
\frac{\mathcal{X}_{-y}(-h_j+w_j-w_i)^{m_j}}{(-h_j+w_j-w_i)^{m_j}}\cdot
\frac{\mathcal{X}_{-y}(h_j-w_j+w_i)}{(h_j-w_j+w_i)}
\right).
\end{align*}
After putting $-q=x_1=\cdots=x_N$ and $s=N-\ell$, this expression can be rewritten as
\begin{align*}
\frac{1}{(1-y)^{s}}\sum\limits_{m_j \geq 0} {\textstyle\prod_j x_j^{m_j}\ \big[\prod_j h_j^{m_j}\big] \prod_j \Phi_j^{m_j} \cdot \Psi_{I\sqcup J}},
\end{align*}
where
\begin{align*}
\Phi_j&=(-h_j) \cdot \prod_{k=1}^N \frac{\mathcal{X}_{-y}(-h_j+w_j-w_k)}{(-h_j+w_j-w_k)},\\
\Psi_{I\sqcup J}&=\prod_j h_j \cdot \prod_{k,j}\frac{\mathcal{X}_{-y}(h_j-w_j+w_k)}{(h_j-w_j+w_k)} \cdot \prod_{j_1 \neq j_2} \frac{(h_{j_1}-h_{j_2}-w_{j_1}+w_{j_2})}{\mathcal{X}_{-y}(h_{j_1}-h_{j_2}-w_{j_1}+w_{j_2})}.
\end{align*}
In the above product, we remind the reader the conventions that elements of $I$, $J$, $[N]$ are denoted by $i$, $j$, $k$ respectively. Using the multivariate Lagrange-B\"{u}rmann formula of \cite{WW}, we find
\[\frac{1}{(1-y)^s}\sum\limits_{m_j \geq 0} {\textstyle\prod_j x_j^{m_j}\ \big[\prod_j h_j^{m_j}\big] \prod_j \Phi_j^{m_j} \cdot \Psi_{I\sqcup J}}
=\frac{1}{(1-y)^s}\cdot \frac{\Psi_{I\sqcup J}}{K_{I\sqcup J}}
\]
for the change of variables
\[q=-x_j=\frac{-h_j}{\Phi_j}=\prod_{k=1}^N \frac{(-h_j+w_j-w_k)}{\mathcal{X}_{-y}(-h_j+w_j-w_k)}
\]
and for
\[K_{I\sqcup J}=\prod_j \left(1-h_j \cdot \frac{\frac{d}{dh_j}\Phi_j}{\Phi_j}\right)
.\]
Since 
\begin{equation}\label{eqngl}
\Bigg(\mathsf{U}_N\ \prod_{i\in I} \mathsf{V}_{N,i}
\prod_{\substack{i_1<i_2\\ i_1, i_2\in I}}\!\!\mathsf{W}_{N,i_1,i_2}\Bigg)^{-1}
=\frac{1}{(1-y)^s}\cdot \frac{\Psi_{I\sqcup J}}{K_{I\sqcup J}}
\end{equation}
only involves variables $h_j$ for $j\in J$, we cannot apply the argument in the proof of Theorem \ref{t0} directly. However, after adding up all such terms corresponding to the partitions $I\sqcup J=[N]$, the same argument will give
\[\sum_{\substack{I\sqcup J=[N], \\ s.t. \ |I|=\ell}} 
\Bigg(\mathsf{U}_N\ \prod_{i\in I} \mathsf{V}_{N,i}
\prod_{\substack{i_1<i_2\\ i_1, i_2\in I}}\!\!\mathsf{W}_{N,i_1,i_2}\Bigg)^{K_X^2}
\in \Q(y, e^{w_1}, \dots, e^{w_N})(q)
\]
without poles at $e^{w_1}=\cdots=e^{w_N}=1$. 

To compute the non equivariant answer for $Z_{X,N,\, \ell K_X}(q)$, we can set up equivariant weights $w_1, \dots, w_N$ to be zero thus work symmetrically in the variables $h_1, \dots, h_N$ as in the proof of Theorem \ref{t0}. If we set equivariant weights to be zero, then
\begin{align*}
\Phi_j&= -h_j\, \left(\frac{1-ye^{h_j}}{1-e^{h_j}}\right)^N,\\
\Psi_{I\sqcup J}&= \prod_j h_j\cdot \prod_j \left(\frac{1-ye^{-h_j}}{1-e^{-h_j}}\right)^N \cdot \prod_{j_1\neq j_2} \frac{1-e^{-h_{j_1}+h_{j_2}}}{1-ye^{-h_{j_1}+h_{j_2}}},\\
K_{I\sqcup J}&=(-1)^s\, N^s\, (1-y)^s \prod_j h_j \cdot \prod_j \frac{e^{h_j}}{(1-e^{h_j})(1-ye^{h_j})}.
\end{align*}
Therefore, (\ref{eqngl}) becomes 
\[\frac{(-1)^s }{N^s (1-y)^{2s}}\,
\prod_j \frac{(1-e^{h_j})(1-ye^{h_j})}{e^{h_j}}\cdot \left(\frac{1-ye^{-h_j}}{1-e^{-h_j}}\right)^N \\
\prod_{j_1\neq j_2}\frac{1-e^{-h_{j_1}+h_{j_2}}}{1-ye^{-h_{j_1}+h_{j_2}}}.
\]
We use the additional change of variables
\[t_k=\frac{1-e^{h_k}}{1-ye^{h_k}}, \ \ \ k=1, \dots, N.\]
This yields the further simplification for the above expression:
\begin{align}\label{B}
\frac{(-1)^{s(N+1)}}{N^s \cdot q^s}
\prod_{j} \frac{t_j (1-(1+y)t_j)^N}{(1-t_j)(1-yt_j)} \times
\prod_{j_1\neq j_2} \frac{t_{j_2}-t_{j_1}}{1-(1+y)t_{j_1}+yt_{j_1}t_{j_2}}
=:\mathsf{A}_{J},
\end{align}
where $t_1, \dots, t_N$ are the distinct roots of the equation $q=t^N$. Note that $A_{\emptyset}=1$ as we take the usual convention that product over empty set is $1$. We have obtained
\begin{equation}\label{C}
\Bigg(\mathsf{U}_N\ \prod_{i\in I} \mathsf{V}_{N,i}
\prod_{\substack{i_1<i_2\\ i_1, i_2\in I}}\!\!\mathsf{W}_{N,i_1,i_2}\Bigg)^{-1}=\mathsf{A}_{J}.
\end{equation}
Therefore, we conclude that 
\begin{align*}
Z_{X,N,\,\ell K_X}(q)&=q^{-\ell K_X^2}\cdot \SW(K_X)^{\ell}
\sum_{\substack{I\sqcup J=[N], \\ s.t. \ |I|=\ell}} 
\left(\mathsf{A}_{J}\right)^{-K_X^2}.
\end{align*}
This completes the proof. 

\subsection{Blow up formula}\label{sbl}
\noindent In this subsection, we prove the blow up formula of the Quot scheme invariants for a curve class $\beta$ of Seiberg-Witten length $N$. The key input for the proof is Lemma \ref{lbl} below. The Lemma follows easily from the blow up formula \cite[Theorem 4.12]{DKO} of the Seiberg-Witten invariants. We denote the pull back of the curve class $\beta$ to the blow up surface by $\widetilde{\beta}$.

\begin{lemma}\label{lbl}
Let $X$ be a smooth projective surface and $\beta$ be a curve class of Seiberg-Witten length $N$. Let $\pi:\widetilde{X}\rightarrow X$ be the blow up of a point with the exceptional divisor $E$. Then $\widetilde{\beta}+\ell E$ is also of Seiberg-Witten length $N$. Furthermore, an effective decomposition
 \[\widetilde{\beta}+\ell E=\sum_{i=1}^N (\widetilde{\beta_i}+\ell_i E)\]
satisfies $\SW(\widetilde{\beta_i}+\ell_i E)\neq 0$ for all $i$ only if \,$\vd_{\beta_i}=0$ and $\ell_i=0$ or $1$ for all $i$. 
\end{lemma}
\begin{proof}
Suppose $\beta$ be a curve class of Seiberg-Witten length $N$. Assume that 
 \[\widetilde{\beta}+\ell E=\sum_{i=1}^N (\widetilde{\beta_i}+\ell_i E)\]
 is an effective decomposition satisfying $\SW(\widetilde{\beta_i}+\ell_i E)\neq 0$ for all $i$. By the blow up formula \cite[Theorem 4.12]{DKO} of the Seiberg-Witten invariants, we have
 \[\SW(\widetilde{\beta_i}+\ell_i E)
 =\tau_{\leq 2\cdot\vd_{\widetilde{\beta_i}+\ell_i E}}\big(\SW(\beta_i)\big)
\]
after identifying 
\[ H_{*}\big(\Pic_{\widetilde{X}}^{\widetilde{\beta_i}+\ell_i E}\big)=\Lambda^{*} H^1\big(\widetilde{X},\Z\big)
=\Lambda^*H^1\big(X,\Z\big)=H_*\big(\Pic_X^{\beta_i}\big).
\]
In the formula, $\tau_{\leq n}$ denotes the truncation map up to homological degree $n$. Therefore $\SW(\widetilde{\beta_i}+\ell_i E)\neq 0$ implies $\SW(\beta_i)\neq 0$. Since $\beta=\sum \beta_i$ and $\SW(\beta_i)\neq 0$ for all $i$, we have $\vd_{\beta_i}=0$ for all $i$ by the assumption on $\beta$. On the other hand, 
\[\vd_{\widetilde{\beta_i}+\ell_i E}
=\frac{\left(\widetilde{\beta_i}+\ell_i E\right)\left(\widetilde{\beta_i}+\ell_i E-\widetilde{K_X}-E\right)}{2}=\vd_{\beta_i}-\frac{\ell_i(\ell_i-1)}{2}. 
\]
Since $\vd_{\beta_i}=0$, we have 
\[\vd_{\widetilde{\beta_i}+\ell_i E}=-\frac{\ell_i(\ell_i-1)}{2}\leq 0
\]
with equality if and only if $\ell_i=0$ or $1$. If $\vd_{\widetilde{\beta_i}+\ell_i E}<0$ then $\SW(\widetilde{\beta_i}+\ell_i E)$ is zero for degree reasons. Hence $\widetilde{\beta}+\ell E$ is of Seiberg-Witten length $N$. Furthermore, all integers $\ell_i$ are forced to be either $0$ or $1$.  This proves the second part of the statement. 
\end{proof}

Now we prove the blow up formula stated in Theorem \ref{tbl}. Let $X$ be a smooth projective surface and $\beta$ be a curve class Seiberg-Witten length $N$. Let \[\pi:\widetilde{X}\rightarrow X\] be the blow up of a point with the exceptional divisor $E$. Lemma \ref{lbl} implies that the Quot scheme invariants of $\widetilde{\beta}+\ell E$ vanish unless $0\leq \ell \leq N$. From now, we fix $\ell$ in this range. By Theorem \ref{tmult} and Lemma \ref{lbl}, $Z_{\widetilde{X}, N, \, \widetilde{\beta}+\ell E}(q\,|\,\underline{w})$ is equal to 
\begin{align*}
q^{-(\widetilde{\beta}+\ell E)(\widetilde{K_X}+E)}\!\!\!\!\!\!\!\!\!\!\!\!\!\!\!\!\!\!
\sum_{\substack{\beta=\sum \beta_i,\ \ell=\sum \ell_i\\ 
s.t. \ \vd_{\beta_i}=0,\ \ell_i=0 \textnormal{ or } 1}}\!\!\!\!\!\!\!\!\!\!\!\!\!\!\!\!\!\!
\SW(\widetilde{\beta_1}+\ell_1 E)\cdots\SW(\widetilde{\beta_N}+\ell_N E)\cdot
\widehat{Z}_{X, N,\,(\widetilde{\beta_i}+\ell_iE)_i}(q\,|\, \underline{w}).
\end{align*}
Using the fact that $\SW(\widetilde{\beta_i}+\ell_iE)=\SW(\beta_i)$ for $\ell_i=0$ or $1$, the above formula simplifies to
\begin{equation}\label{eqnbl}
q^\ell\cdot q^{-\beta.K_X}
\sum_{\substack{I\sqcup J=[N]\\s.t. \ |I|=\ell\ }}
\sum_{\substack{\beta=\sum \beta_i\\ \ s.t.\ \vd_{\beta_i}=0}}
\SW(\beta_1)\cdots\SW(\beta_N)\cdot
\widehat{Z}_{X, N,\,(\widetilde{\beta_i}+\delta_{i,I} E)_i}(q\,|\, \underline{w}),
\end{equation}
where $\delta_{i,I}$ is the indicator function of the set $I$. 
Recall from Theorem \ref{tmult} that
\[\widehat{Z}_{X, N,\,(\widetilde{\beta_i}+\delta_{i,I} E)_i}(q\,|\, \underline{w})
=\mathsf{U}_N^{\big(\widetilde{K_X}+E\big)^2}\!
\prod_{i=1}^N\mathsf{V}_{N,i}^{\, \big(\widetilde{\beta_i}+\delta_{i,I} E\big)\big(\widetilde{K_X}+E\big)}\!\!\!\!\!\!
\prod_{1\leq i<j\leq N}\!\!\!\!\!\mathsf{W}_{N,i,j}^{\, \big(\widetilde{\beta_i}+\delta_{i,I} E\big)\big(\widetilde{\beta_j}+\delta_{j,I} E\big)}\!\!.
\]
The natural decomposition of the intersection form on the blow up surface
\[H^2(\widetilde{X},\Z)=H^2(X,\Z)\oplus \Z[E]
\quad \textnormal{with} \quad H^2(X,\Z)\perp \Z[E],\ \ E^2=-1
\]
implies
\begin{align*}
\widehat{Z}_{X, N,\,(\widetilde{\beta_i}+\delta_{i,I} E)_i}(q\,|\, \underline{w})
=\mathsf{Bl}_{I\sqcup J} \cdot \widehat{Z}_{X,N,\,\underline{\beta}}(q\,|\, \underline{w}),
\end{align*}
where 
\[\mathsf{Bl}_{I\sqcup J}:=\Bigg(\mathsf{U}_N\ \prod_{i\in I} \mathsf{V}_{N,i}
\prod_{\substack{i_1<i_2\\ i_1, i_2\in I}}\!\!\mathsf{W}_{N,i_1,i_2}\Bigg)^{-1}.
\]
Therefore (\ref{eqnbl}) becomes
\begin{align*}
&q^\ell\cdot q^{-\beta.K_X}
\sum_{\substack{I\sqcup J=[N]\\s.t. \ |I|=\ell\ }}
\sum_{\substack{\beta=\sum \beta_i\\ \ s.t.\ \vd_{\beta_i}=0}}
\SW(\beta_1)\cdots\SW(\beta_N)\cdot
\mathsf{Bl}_{I\sqcup J} \cdot \widehat{Z}_{X,N,\,\underline{\beta}}(q\,|\, \underline{w})\\
&=\bigg(q^{\ell} \cdot \sum_{\substack{I\sqcup J=[N]\\s.t. \ |I|=\ell}}\mathsf{Bl}_{I\sqcup J}\bigg)
\cdot\  q^{-\beta.K_X} \!\!\!\!\!\!\!\sum_{\substack{\beta=\sum \beta_i\\s.t.\ \vd_{\beta_i}=0}}
\SW(\beta_1)\cdots\SW(\beta_N)\cdot \widehat{Z}_{X,N,\,\underline{\beta}}(q\,|\, \underline{w})\\
&=\bigg(q^{\ell} \cdot \sum_{\substack{I\sqcup J=[N]\\s.t. \ |I|=\ell}}\mathsf{Bl}_{I\sqcup J}\bigg)
\cdot Z_{X,N,\, \beta}(q\,|\,\underline{w}).
\end{align*}
Note that neither 
\[\sum_{\substack{I\sqcup J=[N]\\s.t. \ |I|=\ell}}\mathsf{Bl}_{I\sqcup J}\quad
\textnormal{nor}\quad Z_{X,N,\, \beta}(q\,|\,\underline{w})
\]
have poles at $w_1=\cdots=w_N=0$. Therefore, the generating series of the non equivariant Quot scheme invariants is simply 
\[Z_{\widetilde{X}, N, \, \widetilde{\beta}+\ell E}(q)=
\bigg(q^{\ell} \cdot \sum_{\substack{I\sqcup J=[N]\\s.t. \ |I|=\ell}}\mathsf{Bl}_{I\sqcup J}\bigg)\bigg\vert_{w_1=\cdots=w_N=0} \cdot Z_{X,N,\, \beta}(q)
\]
by taking $w_1=\cdots=w_N=0$ on each factor. 
Recall that 
\[\bigg(\sum_{\substack{I\sqcup J=[N]\\s.t. \ |I|=\ell}}\mathsf{Bl}_{I\sqcup J}\bigg)\bigg\vert_{w_1=\cdots=w_N=0}=\sum_{\substack{I\sqcup J=[N]\\s.t. \ |I|=\ell}} \mathsf{A}_{J}.
\]
was established in (\ref{C}). This completes the proof. 

\subsubsection{Computation of $\mathsf{Bl}_{N,N}$ and $\mathsf{Bl}_{N,N-1}$.}
We explicitly calculate $\mathsf{Bl}_{N,\ell}$ when $\ell=N$ or $\ell=N-1$. It is clear from the formula
\[\mathsf{Bl}_{N,\,\ell}=\sum\limits_{\substack{J\subseteq[N]\\ |J|=N-\ell}} \mathsf{A}_{J}
\] 
that $\mathsf{Bl}_{N,N}=1$ because $J=\emptyset$ and $s=0$. Assume that $\ell=N-1$ and $s=1$. Since $J$ is a singleton, formula (\ref{C}) becomes
\[\mathsf{Bl}_{N,N-1}=\frac{(-1)^{N+1}}{N \cdot q} \sum_{j=1}^N\frac{t_j (1-(1+y)t_j)^N}{(1-t_j)(1-yt_j)},
\]
where $t_1, \dots, t_N$ are distinct roots of $t^N=q$. Using the following fact about roots of unity
\[\sum_{j=1}^N t_j^p=
    \begin{cases}
      Nq^k, & \text{if}\ p=Nk \\
      \ \ 0\ \ , & \text{otherwise,}
    \end{cases}
\]
we obtain
\begin{align*}
\mathsf{Bl}_{N,N-1}
=\frac{(-1)^{N+1}}{N\cdot q} \sum_{k\geq 1} 
Nq^k\, [t^{Nk}]\,\frac{t(1-(1+y)t)^N}{(1-t)(1-yt)}.
\end{align*}
For convenience, let
\begin{equation}\label{D}
\sum_{i=0}^N a_i t^i=\big(1-(1+y)t\big)^N.
\end{equation}
Then
\begin{align*}
[t^{Nk}]\ \frac{t(1-(1+y)t)^N}{(1-t)(1-yt)}
&=[t^{Nk-1}]\ \left(\sum_{i=0}^Na_it^i\right) 
\left(\sum_{n\geq 0} \frac{1-y^{n+1}}{1-y}t^n\right)\\
&=\sum_{i=0}^N \left(a_i\cdot \frac{1-y^{kN-i}}{1-y}\right)
\end{align*}
for $k\geq 1$. Applying this to the above formula, 
\begin{align*}
\mathsf{Bl}_{N,N-1}
&=\frac{(-1)^{N+1}}{q}\sum_{i=0}^N a_i \sum_{k\geq 1}\frac{1-y^{kN-i}}{1-y}q^k\\
&=\frac{(-1)^{N+1}}{(1-y)q}\sum_{i=0}^Na_i\left(\frac{q}{1-q}-\frac{y^{N-i}q}{1-y^Nq}\right)\\
&=\frac{(-1)^{N+1}}{1-y}\Bigg(\frac{\sum_i a_i}{1-q}-\frac{y^{N}\cdot \sum_i a_iy^{-i}}{1-y^Nq}\Bigg)\\
&=\frac{1}{1-y}\cdot\frac{(1-y^N)-(1-y^{2N})q}{(1-q)(1-y^Nq)},
\end{align*}
where (\ref{D}) was used in the last line. This completes the proof. 

\subsection{Surfaces with $p_g>0$ and rationality of generating series}\label{srat} 
\noindent In this subsection, we prove Theorem \ref{trat} regarding the rationality of the generating series for surfaces with $p_g>0$. This follows from Theorems \ref{t0}, \ref{tell}, \ref{tg}, \ref{tbl} together with the understanding of the Seiberg-Witten basic classes. 

We first reduce the proof of Theorem \ref{trat} to the case of minimal surfaces with $p_g>0$. Let $X$ be a surface with $p_g>0$. Let $\pi:\widetilde{X}\rightarrow X$ be the blow up of a point with the exceptional divisor $E$. Note that any curve class of $\widetilde{X}$ is uniquely written as
\[\widetilde{\beta}+\ell E\]
for some curve class $\beta$ of $X$ and $\ell\in\Z$. By Remark \ref{rSW}, $\beta$ is of Seiberg-Witten length $N$ for arbitrary $N$. Theorem \ref{tbl} thus applies to this setting. We obtain
\[Z_{\widetilde{X},\,N,\,\widetilde{\beta}+\ell E}(q)
=\left(q^\ell\cdot \mathsf{Bl}_{N,\,\ell}(q)\right) Z_{X,\,N,\,\beta}(q),
\]
where $\mathsf{Bl}_{N,\,\ell}(q)$ is a rational function in the $q$ variable. Since every surface is a successive blow up of a certain minimal surface, it is enough to prove rationality of $Z_{X,N,\,\beta}(q)$ for minimal surfaces $X$ with $p_g>0$. 

By the Enriques–Kodaira classification, minimal surfaces with $p_g>0$ are of the form:
\begin{enumerate}
\item[(i)] \kod=0: $K3$ or abelian surfaces
\item[(ii)] \kod=1: minimal elliptic surfaces with $p_g>0$
\item[(iii)] \kod=2: minimal surfaces of general type with $p_g>0$.
\end{enumerate} 
In each case, the Seiberg-Witten basic classes are explicitly known by the proof of \cite[Proposition 4.20]{DKO}. For $K3$ or abelian surfaces, the only Seiberg-Witten basic class is $0=K_X$. For minimal elliptic surfaces with $p_g>0$, every Seiberg-Witten basic class is a rational multiple of the fiber. For minimal surfaces of general type with $p_g>0$, the only Seiberg-Witten basic classes are $0$ and $K_X$. In each case, rationality follows by Theorems \ref{t0}, \ref{tell}, \ref{tg}, respectively. This proves Theorem \ref{trat}.

\section{The reduced invariants of $K3$ surfaces}\label{sK3}

\noindent If X is a $K3$ surface, the only Seiberg-Witten basic class is $0=K_X$ by the proof of \cite[Proposition 4.20]{DKO}. Therefore the usual Quot scheme invariants with $\beta\neq 0$ will vanish because $\beta$ cannot be decomposed into Seiberg-Witten basic classes. Even when $\beta=0$, the Quot scheme invariants vanish, unless $n=0$. To obtain nontrivial invariants, we consider the reduced obstruction theory obtained by removing a trivial factor from the obstruction bundle. Thus, we study the reduced invariants of $K3$ surfaces when $N=1$.

Let $X$ be a $K3$ surface and $\beta$ be a curve class of genus $2g-2=\beta^2$ which is big and nef. Then the corresponding line bundle $\OO(D)$ has vanishing higher cohomology groups. By Lemma \ref{lstructure}, we identify
\[\Quot_X(\C^1\!,\beta,n)=X^{[m]}\times\PP^{g},
\]
where $m=n+(g-1)$ and $\PP^{g}=\Hilb_X^\beta=|\OO(D)|$. Note that $\Hilb_X^{\beta}$ has a trivial rank $1$ obstruction space
\[H^1(\OO_D(D))\cong H^2(\OO_X),
\]
where the isomorphism follows from the long exact sequence corresponding to 
\[0\rightarrow\OO_X\rightarrow\OO_X(D)\rightarrow\OO_D(D)\rightarrow0.\] 
Therefore, by formula (\ref{vtan}) the virtual tangent bundle of the Quot scheme is
\begin{align*}
T_\Quot ^\vir&=T_{X^{[m]}}+T^\vir_{\Hilb_X^\beta}-R\HHom_\pi(\OO_\ZZ,\OO(\DD))\\
&=T_{X^{[m]}}+T_{\PP^{g}}-H^2(\OO_X)-\big((K_X-D)^{[m]}\big)^\vee\otimes \LL,
\end{align*}
where $\LL=\OO_{\PP^g}(1)$ denotes the tautological line bundle over $\PP^{\, g}$. This implies that the obstruction bundle of $\Quot_X(\C^1\!,\beta,n)$ equals
\[ H^2(\OO_X)+\big((-D)^{[m]}\big)^\vee\otimes \LL
\]
By removing a trivial factor $H^2(\OO_X)$ from the obstruction theory, we get a reduced obstruction bundle 
\[\red.\Obs=\big((-D)^{[m]}\big)^\vee\otimes \LL.
\]
inducing the reduced obstruction theory. The reduced virtual $\chi_{-y}$-genus is expressed as an integral
\[\chi_{-y}^\red\big(\Quot_X(\C^1\!,\beta,n)\big)=
\int_{X^{[m]}\times {{\PP\phantom{\vert\!\!}}^g}}e\big(((-D)^{[m]})^\vee\otimes \mathcal{L}\big)\cdot\frac{\mathcal{X}_{-y}(T_{X^{[m]}})\mathcal{X}_{-y}(T_{{{\PP\phantom{\vert\!\!}}^g}})}
{\mathcal{X}_{-y}\big(((-D)^{[m]})^\vee\otimes \mathcal{L})\big)}.
\]
\subsection{$K3$ surfaces with primitive curve classes}\label{proof of K3}
\noindent Now we prove Theorem \ref{tK3}. We closely follow the proof of \cite[Theorem 21]{OP}. No changes are needed here other than exchanging the total Chern class $c(\cdot)$ to the class $\mathcal{X}_{-y}(\cdot)$.  Since $X^{[m]}$ is holomorphic symplectic, we can replace the tangent bundle with the isomorphic cotangent bundle in the above formula
\[\chi_{-y}^\red\big(\Quot_X(\C^1\!,\beta,n)\big)=
\int_{X^{[m]}\times {{\PP\phantom{\vert\!\!}}^g}}e\big(((-D)^{[m]})^\vee\otimes \mathcal{L}\big)\cdot\frac{\mathcal{X}_{-y}\big((T_{X^{[m]}})^\vee\big)\mathcal{X}_{-y}(T_{{\PP\phantom{\vert\!\!}}^g})}
{\mathcal{X}_{-y}\big(((-D)^{[m]})^\vee\otimes \mathcal{L})\big)}.
\]
We claim that 
\begin{multline*}
\int_{X^{[m]}\times {{\PP\phantom{\vert\!\!}}^g}}e\big(((-D)^{[m]})^\vee\otimes \mathcal{L}\big)\cdot\frac{\mathcal{X}_{-y}\big((T_{X^{[m]}})^\vee\big)\mathcal{X}_{-y}(T_{{\PP\phantom{\vert\!\!}}^g})}
{\mathcal{X}_{-y}\big(((-D)^{[m]})^\vee\otimes \mathcal{L})\big)}\\
=\int_{X^{[m]}\times {{\PP\phantom{\vert\!\!}}^g}}e(D^{[m]}\otimes \mathcal{L})\cdot\frac{\mathcal{X}_{-y}(T_{X^{[m]}})\mathcal{X}_{-y}(T_{{\PP\phantom{\vert\!\!}}^g})}
{\mathcal{X}_{-y}(D^{[m]}\otimes \mathcal{L})}.
\end{multline*}
After integrating out the hyperplane class $c_1(\mathcal{L})=h$ on $\PP^g$, this claim is equivalent to the statement 
\[\int_{X^{[m]}} \mathsf{P}\left(c_i\left(\left((-D)^{[m]}\right)^\vee\right), c_j\left(\left(T_{X^{[m]}}\right)^\vee\right)\right)
=
\int_{X^{[m]}} \mathsf{P}\left(c_i\left(D^{[m]}\right), c_j\left(T_{X^{[m]}}\right)\right),
\]
where $\mathsf{P}$ is a universal polynomial in the Chern classes of tautological bundles on $X^{[m]}$ with coefficients in $\Q(y)$. Since $X^{[m]}$ is even dimensional, we can remove the duals on the left hand side. Therefore we must show
\[\int_{X^{[m]}} \mathsf{P}\left(c_i\left((-D)^{[m]}\right), c_j\left(T_{X^{[m]}}\right)\right)
=
\int_{X^{[m]}} \mathsf{P}\left(c_i\left(D^{[m]}\right), c_j\left(T_{X^{[m]}}\right)\right).
\]
By using the argument in \cite{EGL}, the left hand side is the universal polynomial in the Chern numbers
\[(-D)^2, \ K_X^2, \ (-D).K_X, \ c_2(X)\]
with coefficients in $\Q(y)$. The right hand side is similar, with the relevant Chern numbers being
\[D^2, \ K_X^2, \ D.K_X, \ c_2(X).\]
Since $X$ is a $K3$ surface, $K_X$ is trivial hence all the Chern numbers match. This proves the claim. 

To complete the proof we need to show that 
\[\chi_{-y}(\mathcal{C}_g^{[m]})=
\int_{X^{[m]}\times {{\PP\phantom{\vert\!\!}}^g}}e(D^{[m]}\otimes \mathcal{L})\cdot\frac{\mathcal{X}_{-y}(T_{X^{[m]}})\mathcal{X}_{-y}(T_{{\PP\phantom{\vert\!\!}}^g})}
{\mathcal{X}_{-y}(D^{[m]}\otimes \mathcal{L})}.
\]
It was noted in \cite[Section 4]{KST} that the tautological section of $D^{[m]}\otimes \mathcal{L}$ on $X^{[m]}\times \PP^g$ has a scheme theoretic zero locus 
\[\mathcal{C}_g^{[m]} \subset X^{[m]}\times \PP^g.\]
On the other hand, $\mathcal{C}_g^{[m]}$ is smooth of expected dimension when $\beta$ is an irreducible class by \cite{KY}.  
In general, smooth projective zero locus $Z\subset M$ cut out by a regular section of a vector bundle $E\to M$ satisfies
\[\chi_{-y}(Z)= \int_{M} e(E)\cdot\frac{\mathcal X_{-y}(T_M)}{\mathcal{X}_{-y}(E)}.\] 
Applying this to our setting, we obtain the formula claimed above.

\begin{remark}\normalfont
One can modify the proof of Theorem \ref{tK3}, paying attention to the Picard group contribution, to obtain a similar statement about abelian surfaces: if $X$ is an abelian surface and $\beta$ is an irreducible curve class of genus $2g-2=\beta^2$ which is big and nef, then 
\[\chi_{-y}^\red\big(\Quot_X(\C^1\!,\beta,n)\big)=\chi_{-y}(\mathcal{C}_g^{[m]}),
\] 
where $\mathcal{C}_g$ is the universal curve over $\Hilb_X^\beta$. 
Unfortunately, the right hand side is zero because the abelian surface $X$ itself acts on $\mathcal{C}_g^{[m]}$ by translation without fixed points. To obtain nontrivial invariants, one may restrict to curves in a fixed linear system. It is then natural to ask whether one can define a perfect obstruction theory on the Quot scheme with fixed determinant whose reduced virtual invariants coincide with stable pair invariants. We wish to come back to this problem in future work. 
\end{remark}
Now we prove Corollary \ref{cor} by combining Theorem \ref{tK3} with the Kawai-Yoshioka formula of \cite{KY}. Let $X$ be a $K3$ surface and $\beta$ be a primitive curve class which is big and nef of arithmetic genus $g$. By deformation invariance of reduced invariants, we may assume that $\beta$ is irreducible. By Theorem \ref{tK3}, the generating series of the shifted reduced invariants is
\begin{align*}
\overline{Z}_{X,1,\,\beta}^\red(t)&:=\sum_{n\in \Z} \overline{\chi}_{-y}^\red\left(\Quot_X(\C^1\!,\beta,n)\right) t^n=\sum_{m\geq 0} \overline{\chi}_{-y}(\mathcal{C}_g^{[m]}) t^{m+1-g}
\end{align*}
since 
\[\red.\vd\big(\Quot_X(\C^1\!,\beta,n)\big)=m+g=\dim(\mathcal{C}_g^{[m]}).
\]
Recall the specialization of the Kawai-Yoshioka formula to the shifted ${\chi}_{-y}$-genus
\begin{equation}\label{KY}
\sum_{g\geq 0}\sum_{m\geq 0} \overline{\chi}_{-y}(\mathcal{C}_g^{[m]})t^{m+1-g}q^{g-1}
=\frac{y^{-1/2}-y^{1/2}}{\Delta(q)}\frac{\theta'(1)^3}{\theta(y^{1/2}/t)\theta(ty^{1/2})\theta(y)}. 
\end{equation}
Here, the modular discriminant $\Delta(q)$, the theta function $\theta(y,q)=\theta(y)$ and the derivative $\theta'(1)$ are given by
\begin{align*}
\Delta(q)&:=q\prod_{n>0} (1-q^n)^{24},\\
\theta(y)&:=q^{1/8}(y^{1/2}-y^{-1/2})\prod_{n>0}(1-q^n)(1-q^ny)(1-q^n/y),\\
\theta(1)'\!&:=\left(y\frac{\partial}{\partial y}\,\theta(y)\right)\bigg\vert_{y=1}=q^{1/8}\prod_{n>0} (1-q^n)^3.
\end{align*}
Substituting these into (\ref{KY}), we obtain
{\small
\begin{multline*}
\frac{1}{(t+\frac{1}{t}-\sqrt{y}-\frac{1}{\sqrt{y}})\cdot q\ }\ \times\\
\prod\limits_{n>0} \frac{1}{(1-q^n)^{18}(1-q^ny)(1-q^n/y)(1-q^n\sqrt{y}/t)(1-q^nt/\sqrt{y})(1-q^nt\sqrt{y})(1-q^n/t\sqrt{y})
}.
\end{multline*}}
Therefore the corollary follows from 
\[\overline{Z}_{X,1,\,\beta}^\red(t)=[q^{g}]\left(q\cdot \sum_{g\geq 0}\sum_{m\geq 0} \overline{\chi}_{-y}(\mathcal{C}_g^{[m]})t^{m+1-g}q^{g-1}\right).
\]
\subsection{$K3$ surfaces with the zero curve class}\label{sred}
\noindent In this subsection, we prove Theorem \ref{tred}. Let $X$ be a $K3$ surface. The obstruction bundle $(\mathcal{O}^{[n]})^\vee$ of the Hilbert scheme of points $X^{[n]}$ has a trivial factor for $n\geq 1$. To obtain non trivial invariants, we use the reduced obstruction theory with the reduced obstruction bundle
\[ \red.\Obs=\big(\mathcal{O}^{[n]}-\mathcal{O}\big)^\vee.\]

Theorem \ref{tred} does not involve any curve class $\beta$, but the proof will make use of Theorem \ref{tK3}. By the deformation invariance of the reduced $\chi_{-y}$-genus, we may assume that $X$ is an elliptic $K3$ surface with the fiber class $f$. The idea of using elliptic $K3$ surface is due to \cite{OP} which we closely follow below. Note that the curve class $f$ is of genus 1 and the associated line bundle has no higher cohomology groups. Therefore the Quot scheme  $\Quot_X(\C^1\!,f, n)$ is identified with
\[X^{[n]}\times \PP^1.\]
We claim that 
\begin{equation}\label{E}
\chi_{-y}^\red (X^{[n]})=\chi_{-y}^\red\big(\Quot_X(\C^1\!,f, n)\big)\quad  \text{for} \ \ n\geq 1.
\end{equation}
First, we show that this claim implies Theorem \ref{tred}. Note that the fiber class $f$ is not big, but the same proof of Theorem \ref{tK3} will work; bigness was used only to conclude that $\beta$ has no higher cohomology groups, which is the case for us. 
Since the reduced virtual dimension of $\Quot_X(\C^1\!,f, n)$ is $n+1$, 
\begin{align*}
\sum_{n=1}^\infty t^n \chi_{-y}^\red (X^{[n]})
&=\sum_{n=1}^\infty t^ny^{\frac{n+1}{2}} \overline{\chi}_{-y}^\red\big(\Quot_X(\C^1\!,f, n)\big)\\
&=\sqrt{y} \sum_{n=1}^\infty (t\sqrt{y})^n\, \overline{\chi}_{-y}^\red\big(\Quot_X(\C^1\!,f, n)\big).
\end{align*}
By Theorem \ref{tK3}, this is the coefficient of $q^1$ of the expression
\begin{multline*}
\frac{1}{(t+\frac{1}{t}-\sqrt{y}-\frac{1}{\sqrt{y}})\ }\ \times\\
\prod\limits_{n>0} \frac{1}{(1-q^n)^{18}(1-q^ny)(1-q^n/y)(1-q^n\sqrt{y}/t)(1-q^nt/\sqrt{y})(1-q^n\sqrt{y}t)(1-q^n/\sqrt{y}t)}
\end{multline*}up to minor modifications: subtract the coefficient of $t^0$, change the variable $t$ by $t\sqrt{y}$, and multiply by $\sqrt{y}$. One can check by computation that this gives Theorem \ref{tred}. 

Now we prove (\ref{E}) by analyzing the integral giving the reduced $\chi_{-y}$-genus in detail. The reduced $\chi_{-y}$-genus of $X^{[n]}$ can be written as
\[\chi_{-y}^\red(X^{[n]})
=\int_{X^{[n]}} e\left(\big(\mathcal{O}^{[n]}-\mathcal{O}\big)^\vee\right)\cdot 
\frac{\mathcal{X}_{-y}(T_{X^{[n]}})}{\mathcal{X}_{-y}\left(\big(\mathcal{O}^{[n]}-\mathcal{O}\big)^\vee\right)}.
\]
Denote the Chern roots of the tautological bundle $\mathcal{O}^{[n]}$ by $\alpha_1, \cdots, \alpha_n$ with $\alpha_1=0$. Then the above integral becomes
\begin{equation}\label{F}
\int_{X^{[n]}} \prod_{i=2}^n \frac{1-e^{\alpha_i}}{1-ye^{\alpha_i}} \cdot 
\mathcal{X}_{-y}(T_{X^{[n]}}).
\end{equation}
By the previous subsection, the reduced obstruction bundle of the Quot scheme $\Quot_X(\C^1\!,f, n)=X^{[n]}\times \PP^1$ is
\[\Obs=\left(\mathcal{O}(-f)^{[n]}\right)^\vee\otimes \mathcal{L},
\]
where $\mathcal{L}$ is the tautological line bundle of $\PP^1$. Therefore the reduced $\chi_{-y}$-genus of $\Quot_X(\C^1\!,f, n)$ can be written as
\[\chi_{-y}^\red \big(\Quot_X(\C^1\!,f, n)\big)
=\int_{X^{[n]}\times \PP^1}e\left(\left(\mathcal{O}(-f)^{[n]}\right)^\vee\otimes \mathcal{L}
\right)\cdot 
\frac{\mathcal{X}_{-y}(T_{X^{[n]}})\cdot \mathcal{X}_{-y}(T_{\PP^1})}{\mathcal{X}_{-y}\left(\left(\mathcal{O}(-f)^{[n]}\right)^\vee\otimes \mathcal{L}\right)}.
\]
Denote the Chern roots of the tautological bundle $\mathcal{O}(-f)^{[n]}$ by $\mu_1, \cdots, \mu_n$. Denote the hyperplane class of $\PP^1$ by $h=c_1(\mathcal{L})$. The Euler sequence of the tangent bundle of $\PP^1$ implies
\[\mathcal{X}_{-y}(T_{\PP^1})=\frac{\mathcal{X}_{-y}(\mathcal{O}(1))^2}{1-y}=(1-y)+(1+y)h.
\]
Substituting these, the above integral becomes 
\[\int_{X^{[n]}\times \PP^1}
\prod_{i=1}^n \frac{1-e^{-h+\mu_i}}{1-ye^{-h+\mu_i}} \cdot
\big((1-y)+(1+y)h\big)\cdot \mathcal{X}_{-y}(T_{X^{[n]}}).
\]

By integrating out the hyperplane class over $\PP^1$, the integral  becomes
\[\int_{X^{[n]}} 
\left((1+y)[h^0]\prod_{i=1}^n \frac{1-e^{-h+\mu_i}}{1-ye^{-h+\mu_i}}
+(1-y)[h^1]\prod_{i=1}^n \frac{1-e^{-h+\mu_i}}{1-ye^{-h+\mu_i}}\right)\mathcal{X}_{-y}(T_{X^{[n]}})\]
which simplifies to 
\[\int_{X^{[n]}}\left( (1+y)\prod_{i=1}^n \frac{1-e^{\mu_i}}{1-ye^{\mu_i}}
+(1-y)^2\sum_{i=1}^n \frac{e^{\mu_i}}{(1-ye^{\mu_i})^2}\cdot \prod_{j\neq i} \frac{1-e^{\mu_i}}{1-ye^{\mu_i}}
\right)\mathcal{X}_{-y}(T_{X^{[n]}}).
\]
This is a tautological integral over $X^{[n]}$ of the Chern classes of $\mathcal{O}(-f)^{[n]}$ and $T_{X^{[n]}}$. Therefore, as in \cite{EGL}, the resulting integral is an universal polynomial in the Chern numbers of the involved data
\[M^2, \ K_X^2, \ M.K_X, \ c_2(X) ,\]
where $M=\mathcal{O}(-f)$. These numbers are unchanged when we replace $M$ by $\mathcal{O}$ because $f^2=0$ and $K_X=0$. Hence the integral is unchanged if we replace $\mu_i$ by $\alpha_i$. Since $\alpha_1=0$, the integral becomes
\[\int_{X^{[n]}}\left((1+y)\cdot 0+(1-y)^2 \frac{1}{(1-y)^2} \cdot \prod_{j=2}^n \frac{1-e^{\alpha_j}}{1-ye^{\alpha_j}}\right)\mathcal{X}_{-y}(T_{X^{[n]}}).
\]
because 
\[\frac{1-e^{\alpha_1}}{1-ye^{\alpha_1}}=\frac{\alpha_1}{\XX_{-y}(\alpha_1)}=0.
\]
The last expression equals (\ref{F}), thus establishing claim (\ref{E}).


\begin{thebibliography}{1000000}

\bibitem[BJ]{BJ}
B. Bakker, A. Jorza, {\it Higher rank stable pairs on K3 surfaces},
Commun. Number Theory Phys. {\bf 6} (2012), 805-847.

\bibitem[BHPV]{BHPV}
W. Barth, K. Hulek, C. Peters, A. van de Ven, {\it Compact complex surfaces}, second edition, Ergebnisse der Math. No. 4, (Springer, Berlin, 2004). 

\bibitem[BF]{BF}
K. Behrend, B. Fantechi, {\it The intrinsic normal cone}, Invent. Math. {\bf 128} (1997), 45-88. 

\bibitem[CFK]{CFK}
I. Ciocan-Fontanine, M. Kapranov, {\it Virtual fundamental classes for dg-manifolds}, Geom. Topol. {\bf 13} (2009), 1779-1804.

\bibitem[CK]{CK}
H. Chang, Y. Kiem, {\it Poincaré invariants are Seiberg-Witten invariants}, Geom. Topol. {\bf 17} (2013), 1149-1163.

\bibitem[DKO]{DKO}
M. Dürr, A. Kabanov, C. Okonek, {\it Poincaré invariants}, Topology {\bf 46} (2007), 225-294.

\bibitem[EGL]{EGL}
G. Ellingsrud, L. G\"ottsche, M. Lehn, {\it On the cobordism class of the Hilbert scheme of a surface}, J. Alg. Geom. {\bf 10} (2001), 81-100.

\bibitem[FG]{FG}
B. Fantechi, L. G\"ottsche, {\it Riemann-Roch theorems and elliptic genus for virtually smooth schemes}, Geom. Topol. {\bf 14} (2010), 83-115.

\bibitem[F]{F}
J. Fogarty, {\it Algebraic Families on an Algebraic Surface}, Am. J. Math. {\bf 10} (1968), 511-521.

\bibitem[GSY]{GSY}
A. Gholampour, A. Sheshmani, S. Yau, {\it Nested Hilbert schemes on surfaces: Virtual fundamental class}, 
Adv. Math. {\bf 365} (2020), 107046.

\bibitem[GT]{GT}
A. Gholampour, R. P. Thomas, {\it Degeneracy loci, virtual cycles and nested Hilbert schemes II}, \textsf{arXiv:1902.04128}.

\bibitem[GK]{GK}
L. G\"ottsche, M. Kool, {\it Virtual refinements of the Vafa-Witten formula}, 
Comm. Math. Phys. {\bf 376} (2020) 1-49.

\bibitem[GNY1]{GNY1}
L. G\"ottsche, H. Nakajima, K. Yoshioka, {\it Donaldson = Seiberg-Witten from Mochizuki's formula and instanton counting}, Publ. Res. Inst. Math. Sci. {\bf 47} (2011), 307-359.

\bibitem[GNY2]{GNY2}
L. G\"ottsche, H. Nakajima, K. Yoshioka, {\it Instanton counting and Donaldson invariants}, 
J. Diff. Geom. {\bf 80} (2008), 343-390.

\bibitem[GS]{GS}
L. G\"ottsche, V. Shende, {\it The $\chi_{-y}$-genera of relative Hilbert schemes for linear systems on Abelian and K3 surfaces}, Algebr. Geom. {\bf 2} (2015), 405-421.

\bibitem[GP]{GP}
T. Graber, R. Pandharipande, {\it Localization of virtual classes}, Invent. Math. {\bf 135} (1999), 487-518.

\bibitem[JOP]{JOP}
D. Johnson, D. Oprea, R. Pandharipande, {\it Rationality of descendent series for Hilbert and Quot schemes of surfaces}, \textsf{arXiv:2002.05861}.

\bibitem[KY]{KY}
T. Kawai, K. Yoshioka, {\it String partition functions and infinite products}, Adv. Theor. Math. Phys. {\bf 4} (2000), 397-485.

\bibitem[KL]{KL}
Y. H. Kiem, J. Li, {\it Localizing Virtual Cycles by Cosections}, J. Amer. Math. Soc. {\bf 26} (2013), 1025-1050.

\bibitem[KST]{KST}
M. Kool, V. Shende, R. P. Thomas, {\it A short proof of the G\"ottsche conjecture}, Geom. Topol. {\bf 15} (2011), 397-406.

\bibitem[KT]{KT}
M. Kool, R. P. Thomas, {\it Reduced classes and curve counting on surfaces II: calculations}, Algebr. Geom. {\bf 1} (2014), 384-399.

\bibitem[L]{L}
T. Laarakker, {\it Monopole contributions to refined Vafa-Witten invariants}, \textsf{arXiv:1810.00385}.

\bibitem[LQ1]{LQ1}
W.-P. Li and Z. Qin, {\it On blowup formulae for the S-duality conjecture of Vafa and Witten}, Invent. Math. {\bf 136} (1999), 451-482.

\bibitem[LQ2]{LQ2}
W.-P. Li and Z. Qin, {\it On blowup formulae for the S-duality conjecture of Vafa and Witten II: the universal functions}, Math. Res. Lett. {\bf 5} (1998), 439-453.

\bibitem[LT]{LT}
J. Li, G. Tian, {\it Virtual moduli cycles and Gromov-Witten invariants of algebraic varieties}, J. Amer. Math. Soc. {\bf 11} (1998), 119-174.

\bibitem[MO]{MO}
A. Marian, D. Oprea, {\it Virtual intersections on the Quot scheme and Vafa-Intriligator formulas}, Duke Math. Journal {\bf 136} (2007), 81-113.

\bibitem[MOP]{MOP}
A. Marian, D. Oprea, R. Pandharipande, {\it Segre classes and Hilbert schemes of points}, Ann. Sci. ENS. {\bf 50} (2017), 239-267.

\bibitem[M]{M}
T. Mochizuki, {\it Donaldson type invariants for algebraic surfaces}, Lecture Notes in Math. 1972,
Springer-Verlag, Berlin (2009).

\bibitem[OP]{OP}
D. Oprea, R. Pandharipande, {\it Quot schemes of curves and surfaces: virtual classes, integrals, Euler characteristics}, \textsf{arXiv:1903.08787}.

\bibitem[S]{S}
B. Siebert, {\it Virtual fundamental classes, global normal cones and Fulton's canonical classes},  in: Frobenius manifolds, ed. K. Hertling and M. Marcolli, Aspects Math. {\bf 36}
 Vieweg (2004), 341-358.

\bibitem[Ta1]{Ta1}
H. Taubes, {\it Gr $\Rightarrow$ SW: from pseudo-holomorphic curves to Seiberg-Witten solutions}, J. Diff. Geom. {\bf 51} (1999), 203-334.

\bibitem[Ta2]{Ta2}
H. Taubes, {\it SW $\Rightarrow$ Gr: from the Seiberg-Witten equations to pseudo-holomorphic curves}, J. Amer. Math. Soc. {\bf 9} (1996), 845-918.

\bibitem[T1]{T1}
R. P. Thomas, {\it A holomorphic Casson invariant for Calabi-Yau 3-folds, and bundles on K3 fibrations}, J. Diff. Geom. {\bf 54} (2000), 367–438.

\bibitem[T2]{T2}
R. P. Thomas, {\it Equivariant $K$-theory and refined Vafa-Witten invariants}, Commun. Math. Phys. (to appear).
\bibitem[TT1]{TT1}
Y. Tanaka, R. P. Thomas, {\it Vafa-Witten invariants for projective surfaces I: stable case}, J. Alg. Geom. (to appear).

\bibitem[TT2]{TT2}
Y. Tanaka, R. P. Thomas, {\it Vafa-Witten invariants for projective surfaces II: semistable case}, Pure Appl. Math. Quart. {\bf 13} (2017), 517-562.

\bibitem[VW]{VW}
C. Vafa, E. Witten, {\it A strong coupling test of S-duality}, Nuclear Phys. B {\bf 431} (1994), 3-77.

\bibitem[WW]{WW}
E. Whittaker, G. Watson, {\it A course of modern analysis} (Cambridge University Press, London, 1927) 129-130.


\end{thebibliography}
\end{document}